\documentclass[12pt,reqno]{amsart}
\usepackage{hyperref}
\usepackage{amsmath,amssymb}
\usepackage[english]{babel}
\usepackage{caption}
\usepackage{amssymb,amsmath,euscript,enumerate,tikz}
\usepackage[margin=1in]{geometry}
\usepackage{graphicx}
\usepackage{float}
\usepackage{pgf,tikz}
\usepackage{mathrsfs}
\usepackage{upgreek}

\newcommand \reg{\operatorname{reg}}

\newcommand \Tor{\operatorname{Tor}}

\newcommand \pd{\operatorname{pd}}

\newcommand \ini{\operatorname{in}}

\newcommand \h{\operatorname{ht}}

\newcommand \depth{\operatorname{depth}}
\newcommand \gr{\operatorname{gr}}

\newcommand \R{\mathcal{R}}

\newcommand \K{\mathbb{K}}

\newcommand{\mS}{\mathcal{S}}
\newcommand{\rank}{\operatorname{rank}}
\newcommand{\KK}{\mathbb{K}}
\newcommand{\sk}{{\ensuremath{\sf k }}}

\newtheorem{theorem}{Theorem}[section]
\newtheorem{definition}[theorem]{Definition}
\newtheorem{lemma}[theorem]{Lemma}
\newtheorem{proposition}[theorem]{Proposition}

\newtheorem{question}[theorem]{Question}

\newtheorem{corollary}[theorem]{Corollary}

\newtheorem{con}[theorem]{Conjecture}

\newtheorem{remark}[theorem]{Remark}

\begin{document}
	\title[Almost complete
	intersection binomial edge ideals and their Rees algebras]{Almost complete
		intersection binomial edge ideals and their Rees algebras}
	\author[A. V. Jayanthan]{A. V. Jayanthan}
	\email{jayanav@iitm.ac.in}
	\address{Department of Mathematics, Indian Institute of Technology
		Madras, Chennai, INDIA - 600036}
	
	\author[Arvind Kumar]{Arvind Kumar$^1$}
	\email{arvkumar11@gmail.com}
	\address{Department of Mathematics, Indian Institute of Technology,
	Hauz Khas, New Delhi, INDIA - 110016}
	\thanks{$^1$ The author is funded by National Board for Higher
		Mathematics, India}
	\author[Rajib Sarkar]{Rajib Sarkar$^2$}
	\email{rajib.sarkar63@gmail.com}
	\thanks{$^2$ The author is funded by University Grants Commission,
		India}
	\address{Department of Mathematics, Indian Institute of Technology
		Madras, Chennai, INDIA - 600036}

\begin{abstract}
Let $G$ be a  simple graph on $n$ vertices and $J_G$ denote the
binomial edge ideal of $G$ in the polynomial ring $S = \K[x_1,
\ldots, x_n, y_1, \ldots, y_n].$ In this article, we compute
the second graded Betti numbers of $J_G$, and we obtain a
minimal presentation of it when  $G$ is a tree or a unicyclic
graph. We classify all graphs whose binomial edge
ideals are almost complete intersection, prove that they are generated
by a $d$-sequence and that the Rees algebra of their binomial edge
ideal is Cohen-Macaulay.  We also obtain an explicit description
of the defining ideal of the Rees algebra of those binomial edge
ideals.
\end{abstract}
\maketitle
\section{Introduction}
Let $G$ be a  simple graph with vertex set  $V(G) =[n]:= \{1, \ldots, n\}$ and edge set $E(G)$.
Villarreal in \cite{vill_cohen} defined the edge ideal of $G$ as $I(G)
= ( x_ix_j ~ : ~ \{i, j\} \in E(G)) \subset \K[x_1, 
\ldots, x_{n}]$. Herzog et al. in \cite{HH1} and independently 
Ohtani in \cite{oh} defined the binomial edge ideal of $G$ as $J_G = (
x_i y_j - x_j y_i ~ : i < j \text{ and } \{i,j\} \in E(G)) \subset S=
	\K[x_1, \ldots, x_{n}, y_1, \ldots, y_{n}]$. In the recent past,
	researchers have been trying to understand the connection between
	combinatorial invariants of $G$ and algebraic invariants of $I(G)$
	and $J_G$. While this relation between $G$ and $I(G)$ is well
	explored (see for example \cite{BBH17} and the references therein),
	the connection between the properties of $G$ and $J_G$ are not very
	well understood, see \cite{HH1,JACM, KM3,AR1, MM, KM12} for a partial
	list. It is known that the Rees algebra of an ideal $I$,
	$\mathcal{R}(I) := \oplus_{n\geq0}I^nt^n$, encodes a lot of
	asymptotic properties of $I$. In the case of monomial edge
	ideals, properties of their Rees algebra have been explored by several
	researchers (see \cite{Vill95} and the citations to this paper).
	In \cite{Vill95}, Villarreal described the generators of the defining
	ideal of the Rees algebra of a graph. As a consequence of this, he
	proved that $I(G)$ is of \textit{linear type}, i.e., the Rees algebra is
	isomorphic to the Symmetric algebra, if and only if $G$ is either a
	tree or an odd unicyclic graph.
	However, nothing much is known about the Rees algebra of binomial edge
	ideals. In this article, we initiate such a study. 
	
	An ideal $I$ in a standard graded polynomial ring is said to be \textit{complete
		intersection} if $\mu(I) = \h(I)$, where $\mu(I)$ denotes the
	cardinality
	of a minimal homogeneous generating set of $I$.  It is said to be 
	\textit{almost complete intersection} if $\mu(I) = \h(I) + 1$ and
	$I_{\mathfrak{p}}$ is complete intersection for all minimal primes
	$\mathfrak{p}$ of $I$.
	It is known that for a connected graph $G$, $J_G$ is complete
	intersection if and only if $G$ is a path, \cite{her1}.  
	Rinaldo studied the Cohen-Macaulayness of certain subclasses of
	almost complete intersection binomial edge ideals, \cite{Rin13}.
	In this article, we characterize graphs whose binomial
	edge ideals are almost complete intersections. We prove that these
	are either a subclass of trees or a subclass of unicyclic graphs
	(Theorems \ref{tree-aci}, \ref{graph-aci}). 
	
	Understanding the depth of the Rees algebra and the associated graded ring
	of ideals has been a long studied problem in commutative algebra. If
	an ideal is generated by a regular sequence in a Cohen-Macaulay local
	ring, then the corresponding associated graded ring and the Rees
	algebra are known to be Cohen-Macaulay. In general, computing the
	depth of these blowup algebras is a non-trivial problem. If an ideal
	is almost complete intersection, then the
	Cohen-Macaulayness of the Rees algebra and the associated graded ring
	are closely related by a result of Herrmann, Ribbe and Zarzuela (see
	Theorem \ref{aci-cmrees}).  We prove that the associated graded ring,
	and hence, the Rees algebra of almost complete intersection binomial
	edge ideals are Cohen-Macaulay, (Theorem \ref{uni-cmrees}).  
	
	
	Another problem of interest for commutative algebraists is to compute
	the defining ideal of the Rees algebra. Describing the defining ideal
	not only gives more insight into the structure of the Rees algebra,
	but it also
	helps in understanding other homological properties and invariants
	associated with the Rees algebra. For example, the maximal degree occurring in a minimal generating set of the defining ideal also serves as a lower bound
	for one of the most important homological and computational invariant,
	the Castelnuovo-Mumford regularity. In general, it is quite a hard
	task to describe the defining ideals of Rees algebras. 
	Huneke proved that the defining ideal of the
	Rees algebra of an ideal generated by a $d$-sequence has a linear
generating set, \cite{Hu80} (see \cite{KNR91} for a simple proof). We
show that homogeneous almost complete intersection ideals in
polynomial rings over an infinite field are generated by a
$d$-sequence, Proposition \ref{d-seq-lemma}. As a consequence, we
derive that if  $J_G$ is an
	almost complete intersection ideal, then $J_G$ is generated by a
	$d$-sequence, (Corollary \ref{d-seq-unicyc-tree}). We also prove that
	being almost complete intersection is not a necessary condition for
	the binomial edge ideal to be generated by a
	$d$-sequence, by showing
	that $J_{K_{1,n}}$ is generated by a $d$-sequence (Proposition
	\ref{d-seq-star}). We then describe the defining ideals of
	the Rees algebras of almost complete intersection binomial edge
	ideals, (Corollary \ref{cor-def-eqn}, Remark \ref{rem-def-eqn}).
	
	It is known that for an ideal $I$ of linear type,
	the generators of the defining ideal of the Rees algebra can be
	obtained from the matrix of a minimal presentation of $I$ \cite{HS}.
	For describing the generating set of the defining ideal of Rees algebras,
	we compute a minimal presentation of ideals. In this process, we
	compute the second graded Betti numbers and generators of the second
	syzygy of $S/J_G$ when $G$ is a tree or a unicyclic graph, (Theorems
	\ref{betti-tree} - \ref{syzygy-unicyclic}). Here we do
	not assume that the binomial edge ideal is almost complete
	intersection. 
	
	
	The article is organized as follows. The second section contains all
	the necessary definitions and notation required in the rest of the
	article. In Section 3, we describe the second graded
	Betti numbers and first syzygy of the binomial edge ideal of trees and
	unicyclic graphs. We study the Rees algebra of almost complete
	intersection binomial edge ideals in Section 4.

	\vskip 2mm \noindent
	\textbf{Acknowledgement:} We would like to thank the anonymous
	referee for asking some pertinent questions which allowed to us
	improve some of the results in the initial draft.

	\section{Preliminaries}
	
	Let $G$  be a  simple graph with the vertex set $[n]$ and edge set
	$E(G)$. A graph on $[n]$ is said to be a \textit{complete graph}, if
	$\{i,j\} \in E(G)$ for all $1 \leq i < j \leq n$. The complete graph on
	$[n]$ is denoted by $K_n$. For $A \subseteq V(G)$, $G[A]$ denotes the
	\textit{induced subgraph} of $G$ on the vertex set $A$, that is, for
	$i, j \in A$, $\{i,j\} \in E(G[A])$ if and only if $ \{i,j\} \in
	E(G)$.  For a vertex $v$, $G \setminus v$ denotes the  induced
	subgraph of $G$ on the vertex set $V(G) \setminus \{v\}$.  A subset
	$U$ of $V(G)$ is said to be a \textit{clique} if $G[U]$ is a complete
	graph. A vertex $v$ of $G$ is said to be a \textit{simplicial vertex}
	if $v$ is contained in only one maximal clique. For a vertex $v$,
	$N_G(v) = \{u \in V(G) :  \{u,v\} \in E(G)\}$ denotes the
	\textit{neighborhood} of $v$ in $G$ and  $G_v$ is the graph on the vertex set
	$V(G)$ and edge set $E(G_v) =E(G) \cup \{ \{u,w\}: u,w \in N_G(v)\}$.
	The \textit{degree} of a vertex  $v$, denoted by $\deg_G(v)$, is
	$|N_G(v)|$. A vertex $v$ is said to be a \textit{pendant vertex} if
	$\deg_G(v) =1$. Let $c(G)$ denote the number of components
	of $G$.  A vertex $v$ is called a \textit{cut vertex} of $G$ if $c(G)
	< c(G \setminus v)$.  For an edge $e$ in $G$, $G\setminus e$ is the
	graph on the vertex set $V(G)$ and edge set $E(G) \setminus \{e\}$.
	An edge $e$ is called a \textit{cut edge} if $c(G) < c(G \setminus
	e)$.  Let
	$u,v \in V(G)$ be such that $e=\{u,v\} \notin E(G)$, then we denote by
	$G_e$, the graph on vertex set $V(G)$ and edge set $E(G_e) = E(G) \cup
	\{\{x,y\} : x,\; y \in N_G(u) \; or \; x,\; y \in N_G(v) \}$.  A
	\textit{cycle} is a connected graph $G$ with $\deg_G(v) = 2$ for all $v \in V(G)$. A
	graph is said to be a \textit{unicyclic} graph if it contains exactly
	one cycle as a subgraph. A graph is a \textit{tree} if it does not have a
	cycle. The \textit{girth} of a graph $G$ is the length of a shortest
	cycle in $G$. A \textit{complete bipartite} graph on $m+n$ vertices, denoted by
	$K_{m,n}$, is the graph with the  vertex set $V(K_{m,n}) = \{u_1, \ldots,
	u_m\} \cup \{v_1, \ldots, v_n\}$ and edge set $E(K_{m,n}) = \{ \{u_i, v_j\} ~ :
	~ 1 \leq i \leq m, 1 \leq j \leq n \}$. A \textit{claw} is the
	complete bipartite graph $K_{1,3}$. A claw $\{u,v,w,z\}$ with center
	$u$ is the graph with vertices $\{u,v,w,z\}$ and edges $\{\{u,v\},
	\{u,w\}, \{u,z\}\}$. For a graph $G$, let $\mathcal{C}_G$ denote the
	set of all induced claws in $G$.
	
	Let $R = \KK[x_1,\ldots,x_m]$ be a standard graded polynomial ring over a
	field $\KK$ and $M$ be a finitely generated graded  $R$-module. 
	Let
	\[
	0 \longrightarrow \bigoplus_{j \in \mathbb{Z}} R(-j)^{\beta_{p,j}(M)} 
	\longrightarrow \cdots \longrightarrow \bigoplus_{j \in \mathbb{Z}} R(-j)^{\beta_{0,j}(M)} 
	\longrightarrow M\longrightarrow 0,
	\]
	be the minimal graded free resolution of $M$, where
	$R(-j)$ is the free $R$-module of rank $1$ generated in degree $j$.
	The number $\beta_{i,j}(M)$ is  called the 
	$(i,j)$-th {\it graded Betti number} of $M$.  Then,  the exact sequence $$ \bigoplus_{j \in \mathbb{Z}} R(-j)^{\beta_{1,j}(M)} 
	\rightarrow \bigoplus_{j \in \mathbb{Z}} R(-j)^{\beta_{0,j}(M)} 
	\rightarrow M\rightarrow 0$$ is called the {\em minimal presentation} of $M$.

	Let $G$ be a graph on $[n]$. For an edge $e = \{i,j\}\in E(G)$ with $i <
	j$, we define $f_e = f_{i,j} =f_{j,i}:=x_i y_j - x_j y_i$.
	For  $T \subset [n]$, let $\bar{T} = [n]\setminus T$ and $c_T$
	denote the number of  components of $G[\bar{T}]$. Also, let $G_1,\cdots,G_{c_T}$ be the 
	components of $G[\bar{T}]$ and for every $i$, $\tilde{G_i}$ denote the
	complete graph on $V(G_i)$. Let $P_T(G) := (\underset{i\in T} \cup \{x_i,y_i\}, J_{\tilde{G_1}},\cdots, J_{\tilde{G}_{c_T}}).$
	A set $T\subset [n]$ is said to have the \textit{cut point property} if, for every
	$i \in T,$ $i$ is a cut vertex of  graph $G[\bar{T} \cup \{i\}]$.
	
	We recall some results on the binomial edge ideal from \cite{HH1}
	which are used in the subsequent sections.
	\begin{theorem}\label{tech-thm}
	Let $G$ be a graph on $[n]$. Then, we have the following:
	\begin{enumerate}[(a)]
		\item (Corollary 2.2) $J_G$ is a radical ideal.
		\item (Lemma 3.1) For $T \subset [n]$, $P_T(G)$ is a prime ideal and  $\h(P_T(G))=n+|T|-c_T$.
		\item (Theorem 3.2)  $J_G =  \underset{T \subset [n]}\cap
		P_T(G)$.
		\item (Corollary 3.9) For $T \subset [n]$, $P_T(G)$ is a minimal prime
		of $J_G$ if and only if either $T = \emptyset$ or $T$ has the cut point property.
	\end{enumerate}
	\end{theorem}
	
	\vskip 2mm \noindent
	\textbf{Mapping Cone Construction:} For an edge $e=\{i,j\}\in E(G)$, We consider the following exact sequence:
	\begin{align}\label{main-ses}
	0\longrightarrow \frac{S}{J_{G\setminus e}:f_e}(-2) \stackrel{\cdot f_e}{\longrightarrow} \frac{S}{J_{G\setminus e}} \longrightarrow \frac{S}{J_G} \longrightarrow 0 .
	\end{align}
	By \cite[Theorem 3.7]{FM}, we have $$J_{G\setminus e}:f_e=J_{(G\setminus e)_e}+ (g_{P,t}: P \text{ is a path of length } s+1 \text{ between } i,j \mbox{ and } 0\leq t\leq s ),$$ 
	where for a path $P: i, i_1, \ldots, i_s, j$, $g_{P,0}=y_{i_1}\cdots y_{i_s}$ and for each $1\leq t\leq s$, $g_{P,t}=x_{i_1}\cdots x_{i_t}y_{i_{t+1}}\cdots y_{i_s}.$ 
	Let $(\mathbf{F}.,d^{\mathbf{F}}.)$ and
	$(\mathbf{G}.,d^{\mathbf{G}}.)$ be minimal $S$-free resolutions of
	$S/J_{G\setminus e}$ and $[S/(J_{G\setminus e}:f_e)](-2)$
	respectively. Let $\varphi. : (\mathbf{G}.,d^{\mathbf{G}}.)
	\longrightarrow (\mathbf{F}.,d^{\mathbf{F}}.)$ be the
	complex morphism induced by the multiplication by $f_e$. 
	The mapping cone $(\mathbf{M}(\varphi).,\delta.)$ is an $S$-free resolution of
	$S/J_G$
	such that $(\mathbf{M}(\varphi))_i=\mathbf{F}_i\oplus
	\mathbf{G}_{i-1}$ and the differential maps are
	$\delta_i(x,y)=(d^{\mathbf{F}}_i(x)+\varphi_{i-1}(y),-d^{\mathbf{G}}_{i-1}(y))$
	for $x\in \mathbf{F}_i$ and $y\in \mathbf{G}_{i-1}$. 
	It need not necessarily be a minimal free
	resolution. We refer the reader to \cite{eisenbud95} for more details on the
	mapping cone.
	
	\section{Betti numbers and Syzygy of binomial edge ideals}
	In this section, we describe the first graded Betti numbers and the first
	syzygy of binomial edge ideals of trees and unicyclic graphs. First, we compute the second graded Betti numbers  of $S/J_G$  when $G$ is a tree.
	\begin{theorem}\label{betti-tree}
		Let $G$ be a tree on $[n]$. Then,   \[\beta_2(S/J_G)=	  \beta_{2,4}(S/J_G)={n-1 \choose 2}+\sum_{v \in V(G)} {\deg_G(v) \choose 3}.\]
	\end{theorem}
	\begin{proof}
		We prove this by induction on $n$. If $n=2$, then $G=P_2$, and hence, $J_G$ is complete intersection. Therefore, $\beta_2(S/J_G) = 0$. Hence,   the 	assertion follows. We now assume that $n>2$. Let $e=\{u,v\}$ be an
		edge such that $u$ is a pendant vertex. 
		The  long exact sequence of Tor in degree $j$ component corresponding to the short exact sequence \eqref{main-ses} is:
		\begin{equation}\label{Tor-les}
		\cdots	\rightarrow \Tor_{2,j}^S\left(\frac{S}{J_{G\setminus e}},\K\right)\rightarrow \Tor_{2,j}^S\left(\frac{S}{J_{G}},\K\right) 
		\rightarrow \Tor_{1,j}^S\left(\frac{S}{J_{G\setminus e}:f_e}(-2),\K\right)\rightarrow\cdots
		\end{equation}
		Since $e$ is a cut edge and $u$ is a pendant vertex of $G$,
		$(G\setminus e)_e= (G \setminus u)_v \sqcup \{u\}$. Thus,    it follows
		from \cite[Theorem 3.7]{FM} that $J_{G\setminus e} :f_e =
		J_{(G\setminus u)_v}$. One can observe that
		
		\[\Tor_{1,j}\left(\frac{S}{J_{(G\setminus u)_v}}(-2),\K\right) \simeq
		\Tor_{1,j-2}\left(\frac{S}{J_{(G\setminus u)_v}},\K\right).\] 
		Since $G \setminus e = (G \setminus u) \sqcup \{u\}$,
		$J_{G\setminus e} = J_{G\setminus u}$.
		Therefore, by  induction, we obtain
		\[\beta_{2,4}(S/J_{G
			\setminus e}) = {n-2 \choose 2}+ \sum_{w \in
			V(G)\setminus\{v\}} { \deg_G(w) \choose 3} + {\deg_G(v)-1 \choose
			3}\] and $\beta_{2,j}(S/J_{G\setminus e})=0$ for $j \neq 4$. If $j
		\neq 4$, then \[\Tor_{1,j-2}\left(\frac{S}{J_{(G\setminus
				u)_v}},\K\right)=0.\] Hence, $\beta_{2,j}(S/J_G)=0$, if $j \neq 4$.
		Since $\beta_{2,2}(S/J_{(G\setminus u)_v})=0$ and
		$\beta_{1,4}(S/J_{G\setminus e})=0$, we have $\beta_{2,4}(S/J_G)
		=\beta_{2,4}(S/J_{G\setminus e})+\beta_{1,2}(S/J_{(G \setminus
			u)_v})$. Now, $\beta_{1,2}(S/J_{(G\setminus u)_v})=|E((G\setminus
		u)_v)|=n-2+{\deg_G(v)-1 \choose 2}$. Hence,
		$\beta_2(S/J_G)=\beta_{2,4}(S/J_G)={n-1 \choose 2}+\sum_{v \in
			V(G)} {\deg_G(v) \choose 3}.$
	\end{proof}
	We now describe the first syzygy of binomial edge ideals of trees. To
	compute a minimal generating set of first syzygy, we
	crucially use the knowledge of the Betti numbers of $J_G$. A tree on $[n]$ vertices has $n-1$ edges.
	For
	convenience in writing the list of generators, we need some notation.
	For $A \subseteq [n]$ and $i \in A$, we define
	$p_A(i) = |\{j \in A \mid j \leq i\}|$. The function
	$p_A$ indicates the position of an element in $A$
	when the elements are arranged in the ascending order.
	\begin{theorem}\label{syzygy-tree}
		Let $G$ be a tree on $[n]$ vertices.  Then,  the first syzygy of $J_G$
		is minimally generated by elements of the form
		\begin{enumerate}
			\item[(a)] $f_{i,j}e_{\{k,l\}} - f_{k,l}e_{\{i,j\}}, \text{ where }
			\{i,j\},\{k,l\}\in E(G) \text{ and } \{e_{\{i,j\}} ~ : ~ \{i,j\} \in
			E(G)\}$ is the standard basis of $S(-2)^{n-1}$;
			\item[(b)]  $(-1)^{p_A(j)}f_{k,l}e_{\{i,j\}} + (-1)^{p_A(k)}
			f_{j,l}e_{\{i,k\}} + (-1)^{p_A(l)}f_{j,k}e_{\{i,l\}}, \text{ where
			} A = \{i,j,k,l\} \in \mathcal{C}_G$ $\text{ with center at } i.$
		\end{enumerate}
		
	\end{theorem}
	\begin{proof}
		From Theorem \ref{betti-tree}, we have $\beta_{2}(S/J_G)= \beta_{2,4}(S/J_G)={n-1
			\choose 2}+\sum_{v \in V(G)} {\deg_G(v) \choose 3}$. 
		Therefore, the minimal presentation of $J_G$ is of the form
		\[  S(-4)^{\beta_{2,4}(S/J_G)} \stackrel{\varphi}{\longrightarrow}
		S(-2)^{n-1} \stackrel{\psi}{\longrightarrow} J_{G}\longrightarrow
		0. \]
		Note that $|\mathcal{C}_G|=\sum_{v \in
			V(G)} {\deg_G(v) \choose 3}$. Since $\beta_2(S/J_G) = {n-1\choose 2} +
		|\mathcal{C}_G|$, we index the standard basis of $S(-4)^{\beta_2(S/J_G)}$
		accordingly. Let
		\[\begin{array}{lll}
		\mathcal{S}_1 & = & \{ E_{\{i,j\},\{k,l\}} :
		\{i,j\},\{k,l\}\in E(G), i< j, k < l  \text{ and } (i,j) \gneq_{lex}(k,l) \}\\
		\mathcal{S}_2 &=&  
		\{E^{i}_{\{j,k,l\}} : \{i,j,k,l\}\in
		\mathcal{C}_G ~ \text{with center at }i \big\}
		\end{array}\]
		and $\mathcal{S} = \mathcal{S}_1 \cup \mathcal{S}_2$ denote the
		standard basis of $S(-4)^{\beta_{2}(S/J_G)}$. For a pair of edges
		$\{i,j\}, \{k, l\} \in E(G)$, $f_{i,j}f_{k,l} - f_{k,l}f_{i,j} = 0$ gives
		a relation among the generators of $J_G$.
		Let $\{i,j,k,l\} \in \mathcal{C}_G$ be a claw with center at $i$. 
		Then, it can be easily verified that for $A = \{i,j,k,l\}$,
		\[(-1)^{p_A(j)}f_{k,l}f_{i,j} + (-1)^{p_A(k)}
		f_{j,l}f_{i,k} + (-1)^{p_A(l)}f_{j,k}f_{i,l} = 0, \]
		which gives another relation among the generators of $J_G$. 
		Define the maps $\varphi$ and $\psi$ as follows:
		\[ 
		\begin{array}{ll}
		\varphi \big( E_{\{i,j\},\{k,l\}} \big) & =   f_{i,j}e_{\{k,l\}} -
		f_{k,l}e_{\{i,j\}}; \\
		\varphi\big( E^{i}_{\{j,k,l\}} \big) & = 
		(-1)^{p_A(j)}f_{k,l}e_{\{i,j\}} + (-1)^{p_A(k)}
		f_{j,l}e_{\{i,k\}} + (-1)^{p_A(l)}f_{j,k}e_{\{i,l\}}; \\ 
		\psi(e_{\{i,j\}})& =  f_{i,j},
		\end{array}
		\]
		where $A = \{i,j,k,l\}$.
		Observe that $\varphi(\mathcal{S}_1)$ is the collection of  elements of type
		$(a)$ in the statement of the Theorem and 
		$\varphi(\mathcal{S}_2)$ is the collection of elements of type
		$(b)$.
		Also, for any pair of edges $\{i, j\}, \{k, l\}$ and a claw
		$\{u,v,w,z\}$ with $u$ as a center, we have $\psi\big( \varphi (E_{\{i,j\},\{k,l\}})
		\big)=0$ and $\psi( \varphi( E^u_{\{v,z,w\}}) )=0.$ Since
		$\beta_{2,j} = 0$ for all $j \neq 4$, it follows that the first syzygy
		is generated in degree $4$. Moreover, as
		$\beta_2(S/J_G) = \beta_{2,4}(S/J_G) = |\mathcal{S}|$, to prove the assertion,
		it is enough to prove that the elements of $\varphi(\mathcal{S})$ are
		$\K$-linearly independent, equivalently, the columns of the matrix of $\varphi$ are
		$\K$-linearly independent. For this, note that for each
		$\{i, j\} \in E(G)$, the entries of the corresponding row are the
		coefficients of $e_{\{i,j\}}$ in the expression for the images of
		elements in $\mathcal{S}$ under $\varphi$. The coefficient of
		$e_{\{i,j\}}$ in $\varphi(E_{\{i,j\},\{k,l\}})$ or
		$\varphi(E_{\{k,l\},\{i,j\}})$ is $\pm f_{k,l}$. Moreover, the entry
		will be zero in the column corresponding to
		$\varphi(E_{\{u,v\},\{w,z\}})$ for $\{u,v\} \neq \{i,j\}$ and $\{w,z\}
		\neq \{i,j\}$. Therefore, among the first ${n-1 \choose 2}$ column entries in
		the row corresponding to $e_{\{i,j\}}$, there will be $(n-2)$ non-zero
		entries, namely the binomials corresponding to all the edges
		other than $\{i,j\}$. In $\varphi(E^u_{\{v,w,z\}})$, the coefficient
		of $e_{\{i,j\}}$ is non-zero if and only if either $i = u$ and $j \in
		\{v,w,z\}$ or $j = u$ and $i \in \{v,w,z\}$.
		If $i = u$ and $j = v$ (similarly any one of the other three), then
		the coefficient of $e_{\{i,j\}}$ is $\pm f_{w,z}$. It may be noted here
		that $f_{w,z}$ does not correspond to an edge in $G$. If  $E^{u_1}_{\{v_1,w_1,z_1\}}$ and
		$E^{u_2}_{\{v_2,w_2,z_2\}}$ are two
		distinct basis elements  $\{i,j\}$ in both the
		claws, then $\{u_1,v_1,w_1,z_1\} \setminus \{i,j\} \neq
		\{u_2,v_2,w_2,z_2\} \setminus \{i,j\}$. Hence,   the corresponding
		coefficients of $e_{\{i,j\}}$ in $\varphi(E^{u_1}_{\{v_1,w_1,z_1\}})$
		and $\varphi(E^{u_2}_{\{v_2,w_2,z_2\}})$ will not be the same. From the
		above discussion one concludes that in the row corresponding to
		$e_{\{i,j\}}$, each nonzero entry is of the form $\pm f_{k,l}$ for some $k, l
		\in [n]$, $\{k, l\} \neq \{i, j\}$ and no two are equal. Therefore,
		the entries of this row can be seen as the minimal generating set of binomial edge ideal of a
		graph on $[n]$, possibly different from $G$, and hence, they are
		$\K$-linearly independent. Therefore, the assertion follows.
	\end{proof}
	
	We now study the first graded Betti numbers and syzygy of binomial edge
	ideal of unicyclic graphs.  Let $G$ be a unicyclic graph on the vertex
	set $[n]$ of girth $m$.  First, we compute $\beta_2(S/J_G)$, where $G$
	is a unicyclic graph of girth $3$. 
	\begin{theorem}\label{betti-unicyclic3}
		Let $G$ be a unicyclic graph on $[n]$ of girth $3$. Let $v_1,v_2,v_3$
		be the vertices of the cycle in $G$. Then,   
		\[\beta_2(S/J_G) = \beta_{2,3}(S/J_G)+\beta_{2,4}(S/J_G)=2 + \beta_{2,4}(S/J_G),\]\[ \beta_{2,4}(S/J_G)={n \choose 2}+\sum_{v \in V(G)} {\deg_G(v) \choose 3} -\sum_{i=1,2,3}\deg_G(v_i)+3.\]
	\end{theorem}
	
	\begin{proof}
		We prove this by induction on $n$. By \cite[Theorem 2.2]{KM12},  for any graph
		$G$, $\beta_{2,3}(S/J_G) = 2\sk_3(G)$, where $\sk_3(G)$ denotes the number
		of $K_3$'s appearing in $G$. If $n=3$, then
		$G=K_3$, and hence, the assertion follows from \cite[Theorem 2.1]{KM12}.
		We now assume that
		$n>3$. Let $e=\{u,v\}$ be an edge such that $u$ is a pendant
		vertex. 
		Since $e$ is a cut edge and $u$ is a pendant vertex of $G$,
		$(G\setminus e)_e= (G \setminus u)_v \sqcup \{u\}$. Thus,
		$J_{G\setminus e} :f_e = J_{(G\setminus u)_v}$. By
		\cite[Theorem 2.2]{KM12}, we get $\beta_{2,3}(S/J_{G\setminus e})=2$.
		Therefore, by induction, we get
		\[\beta_{2,4}(S/J_{G\setminus e})= {n-1 \choose 2}+ \sum_{w \in
			V(G)\setminus\{v\}} { \deg_G(w) \choose 3} + {\deg_G(v)-1 \choose
			3}-\sum_{i=1,2,3}\deg_{G\setminus e}(v_i)+3\] and
		$\beta_{2,j}(S/J_{G\setminus e})=0$ for $j > 4$. If $j \neq 4$, then
		$\Tor_{1,j-2}\left(\frac{S}{J_{(G\setminus u)_v}},\K\right)=0.$
		Hence, the long exact sequence (\ref{Tor-les}) gives  that $\beta_{2,j}(S/J_G)=0$, if $j > 4$. Since
		$\beta_{2,2}(S/J_{(G\setminus u)_v})=0$ and
		$\beta_{1,4}(S/J_{G\setminus e})=0$, it follows from the long exact sequence (\ref{Tor-les}) that  $\beta_{2,4}(S/J_G)
		=\beta_{2,4}(S/J_{G\setminus e})+\beta_{1,2}(S/J_{(G \setminus
			u)_v})$. If $v=v_i$ for some $i$,
		then  $\beta_{1,2}(S/J_{(G\setminus u)_v})=|E((G\setminus
		u)_v)|=|E(G)|-1+{\deg_G(v)-1 \choose 2}-1=n-2+{\deg_G(v)-1 \choose
			2}$. Moreover, for this $i$, $\deg_{G\setminus e}(v_i) = \deg_G(v_i) -
		1$. Hence,   we get the required expression for $\beta_{2,4}(S/J_G)$.
		If $v\neq v_i$ for all $i$, then $\beta_{1,2}(S/J_{(G\setminus
			u)_v})=|E((G\setminus u)_v)|=n-1+{\deg_G(v)-1 \choose 2}$. Hence,
		$\beta_{2,4}(S/J_G)={n \choose 2}+\sum_{v \in V(G)} {\deg_G(v)
			\choose 3}-\sum_{i=1,2,3}\deg_G(v_i)+3.$ 
	\end{proof}
	We now compute the second graded Betti numbers of $S/J_G$ when $G$ is a
	unicyclic graph of girth at least $4$.
	\begin{theorem}\label{betti-unicyclic}
		If $G$ is a unicyclic graph on $[n]$ of girth $m\geq 4$, 
		then  \[\beta_2(S/J_G)= \left\{
		\begin{array}{ll}
		\beta_{2,4}(S/J_G), & \text{ if } m = 4, \\
		\beta_{2,4}(S/J_G) + \beta_{2,m}(S/J_G) & \text{ if } m > 4,
		\end{array} \right.\]
		where
		\[\beta_{2,4}(S/J_G)= \left\{
		\begin{array}{ll}
		{n \choose 2}+\sum_{v \in V(G)} {\deg_G(v) \choose 3} + 3 &
		\text{ if } m = 4,\\
		{n \choose 2}+\sum_{v \in V(G)} {\deg_G(v) \choose 3}  &
		\text{ if } m > 4,
		\end{array}\right.\]
		$ \text{ and }\; \beta_{2,m}(S/J_G)=m-1,$ if $m > 4$.
	\end{theorem}
	\begin{proof}
		Let $e=\{u,v\}$
		be an edge of the cycle in $G$. Then,  after removing the edge $e$,
		$G\setminus e$ becomes a tree. Therefore, from Theorem
		\ref{betti-tree}, we have \[\beta_{2}(S/J_{G\setminus
			e})=\beta_{2,4}(S/J_{G\setminus e})={n-1 \choose 2}+ \sum_{w \in
			V(G)\setminus \{u,v\}} {\deg_G(w) \choose 3}+ \sum_{w\in
			\{u,v\}}{\deg_{G}(w)-1 \choose 3}.\]
		Note that $(G\setminus e)_e= ((G \setminus e)_v)_u$. 
		
		It follows from
		\cite[Theorem 3.7]{FM} that $J_{G\setminus e} :f_e =
		J_{((G\setminus e)_v)_u}+I$, where $$I=(g_{P,t}: P : u,i_1,\dots,i_s,v
		\text{ is a path between }u \text{ and } v \text{ in } G \setminus e ~
		\text{and} ~ 0\leq t\leq s).$$
		In $G\setminus e$, there is only one path between $u$ and $v$ and the
		corresponding $g_{P,t}$ has degree $m-2$ for all $t$. Since $\beta_{2,2}(S/(J_{((G\setminus e)_v)_u}+I))=0$ and
		$\beta_{1,4}(S/J_{G\setminus e})=0$, we have $\beta_{2,4}(S/J_G) =\beta_{2,4}(S/J_{G\setminus e})+\beta_{1,2}(S/(J_{((G \setminus e)_v)_u}+I))$.
		For $m = 4$, $I = (y_2y_3, x_2y_3, x_2x_3)$. Therefore,
		\[\beta_{1,2}(S/(J_{(G\setminus e)_e}+I)) = 3 + |E( (G \setminus
		e)_e)| = 3+(n-1) + {\deg_G(v)-1 \choose 2} + {\deg_G(u) - 1 \choose 2}.\]
		Hence, \[\beta_{2,4}(S/J_G) = \beta_{2,4}(S/J_{G\setminus e}) +
		\beta_{1,2}(S/(J_{(G\setminus e)_e}+I)) = {n\choose 2} +
		\sum_{v\in V(G)}{\deg_G(u) \choose 3}+3.\] Also, $\beta_{2,j}(S/J_{G\setminus e})=0$ and $\beta_{1,j-2}(S/(J_{(G \setminus e)_e}+I))=0$, if $j \neq 4$.
		Therefore, $\beta_{2,j}(S/J_G)=0$, if $j \neq 4$.
		Now assume that $m > 4$. Note that for $j \neq 4$, $j\neq m$, 
		\[\Tor_{1,j-2}\left(\frac{S}{J_{((G\setminus e)_v)_u}+I},\K\right)=0
		\text{ and }
		\dim_{\K}\left(\Tor_{1,m-2}\left(\frac{S}{J_{((G\setminus
				e)_v)_u}+I},\K\right)\right)=m-1.\] 
		Hence,   it follows from the long exact sequence (\ref{Tor-les}) that $\beta_{2,j}(S/J_G)=0$, if
		$j \notin \{4,m\}$. Since $\beta_{1,j}(S/J_{G\setminus e})=0$ for
		$j \neq 2$, we  have \[\Tor_{1,m-2}\left(\frac{S}{J_{((G\setminus
				e)_v)_u}+I},\K\right) \simeq
		\Tor_{2,m}\left(\frac{S}{J_{G}},\K\right) .\] Thus,
		for $m > 4$, $\beta_{2,m}(S/J_G) =m-1$.
		Now, $\beta_{1,2}(S/(J_{((G\setminus e)_v)_u})+I)=|E(((G\setminus e)_v)_u)|=n-1+{\deg_G(v)-1 \choose 2}+{\deg_G(u)-1 \choose 2}$.
		Hence, $\beta_{2,4}(S/J_G)={n \choose 2}+\sum_{v \in V(G)} {\deg_G(v) \choose 3}.$
	\end{proof}
	
	Now, we obtain a minimal generating set for the first syzygy of $J_{C_n}$ for $n\geq 4$.
	Let $G = C_n$ be a cycle on $[n]$ with edge set $E(C_n)=\{\{i,i+1\},\{1,n\}: 1\leq i\leq n-1\}$.  
	
	\begin{theorem}\label{syzygy-cycle}
		Let $C_n$ be the cycle on $n$ vertices, $n \geq 4$. Let $\{e_{\{k,k+1\}},e_{\{1,n\}} : 1
		\leq k \leq n-1\}$ denote the standard basis of $S^n$ and $Y=y_1\cdots y_n$.  For $i = 1,
		\ldots, n-1$, define $b_i \in S^n$ as follows:
		\[
		(b_1)_k = \frac{Y}{y_ky_{k+1}} \text{ for }1 \leq k \leq n-1,
		(b_1)_n = \frac{Y}{y_1y_n}, \]
		\[
		\text{ for } 1 \leq i \leq n-2, ~
		(b_{i+1})_k = 
		\begin{cases}
		{(b_i)_k\cdot \frac{x_{i+2}}{y_{i+2}} } & \text{ if }
		k\leq i, \\
		{(b_i)_{i+1}\cdot \frac{x_1}{y_1} } &  \text{ if } k=i+1,\\
		{(b_i)_k\cdot \frac{x_{i+1}}{y_{i+1}} } & \text{ if }k\geq i+2.
		\end{cases}
		\]
		Then,  the first syzygy of $J_{C_n}$is minimally generated by
		\[
		\big\{ f_{k,l}e_{\{i,j\}} - f_{i,j} e_{\{k,l\}}  : \{i,j\},\{k,l\} \in E(C_n)\big\} \bigcup 
		\left\{ \sum_{k=1}^{n-1}
		(b_i)_ke_{\{k,k+1\}} - (b_i)_ne_{\{1,n\}}~ : ~ 1 \leq i \leq n-1\right\}.
		\]
	\end{theorem}
	
	\begin{proof}
		By \cite[Corollary 16]{Zafar}, 
		$\beta_{2,4}(S/J_{C_n})=\left\{
		\begin{array}{ll}
		9 & \text{ if } n = 4;\\
		\binom{n}{2} & \text{ if } n > 4, 
		\end{array}\right.$ $\beta_{2,n}(S/J_{C_n})=n-1$ for $n > 4$ and
		$\beta_{2,j}(S/J_{C_n}) = 0$ for all $j \neq 4,n$. 
		Therefore,  the minimal presentation of $J_{C_n}$ is
		\begin{equation}
		S(-4)^{\binom{n}{2}}\oplus S(-n)^{n-1} \longrightarrow S(-2)^n \longrightarrow J_{C_n}\longrightarrow 0.
		\end{equation}
		Note that $J_{C_n}=J_{P_n}+(f_{1,n})$. Consider the following exact sequence
		$$ 0\longrightarrow \frac{S}{J_{P_n}:f_{1,n}}(-2) \stackrel{\cdot f_{1,n}}{\longrightarrow} \frac{S}{J_{P_n}} \longrightarrow \frac{S}{J_{C_n}} \longrightarrow 0
		$$
		and apply the mapping cone construction. Since $J_{P_n}$ is complete intersection, the Koszul complex $(\mathbf{F}.,d^{\mathbf{F}}.)$ gives the
		minimal free resolution for $S/J_{P_n}$. 
		Let $\{e_{\{i,j\},\{k,l\}} \mid \{i,j\}\neq \{k,l\} \in E(P_n)\}$ denote the
		standard basis of $S^{\binom{n-1}{2}}$ and $\{e_{\{j,j+1\}} \mid 1 \leq j \leq n-1\}$ denote the standard basis of $S^{n-1}$.
		Set $d^{\mathbf{F}}_1(e_{\{j,j+1\}})=f_{j,j+1}$ for $1\leq j\leq n-1$ and $d^{\mathbf{F}}_2(e_{\{i,j\},\{k,l\}})=f_{k,l}e_{\{i,j\}}-f_{i,j}e_{\{k,l\}}$ for $\{i,j\}\neq \{k,l\} \in E(P_n)$.
		It follows from \cite[Theorem 3.7]{FM} that
		$$J_{P_n}:f_{1,n}=J_{P_n}+(y_2\cdots
		y_{n-1},x_2y_3\cdots y_{n-1}, \ldots, x_2\cdots x_{n-1}).$$
		Let $(\mathbf{G}.,d^{\mathbf{G}}.)$ be the minimal resolution of
		$\frac{S}{(J_{P_n}:f_{1,n})}(-2)$
		with the differential maps
		given by $d^{\mathbf{G}}_1(E_{i,i+1})=f_{i,i+1}$ for $1\leq i\leq n-1$ and
		$d^{\mathbf{G}}_1(E_{m})=x_2\cdots x_my_{m+1}\cdots y_{n-1}$ for $1\leq
		m\leq n-1$, where $\{E_{i,i+1},E_m : 1\leq i\leq n-1,1\leq m\leq n-1 \}$
		denotes the standard basis of $G_1$. 
		Clearly the map from $G_0$ to $F_0$ in the mapping cone complex is the
		multiplication by $f_{1,n}$. Define the map $\varphi_1:
		G_1\longrightarrow F_1$ by 
		\[
		\begin{array}{ll}
		\varphi_1(E_{i,i+1})=f_{1,n}e_{\{i,i+1\}} & 1 \leq i \leq n-1,\\
		\varphi_1(E_m) =\sum_{k=1}^{n-1} (b_m)_k e_{\{k,k+1\}} & 1 \leq m \leq n-1,
		\end{array}
		\]
		where $(b_m)_k$'s are as defined in the statement of the Theorem.
		We show that the map $\varphi_1$ satisfies the property that for all
		$x\in G_1,$
		$d^{\mathbf{F}}_1(\varphi_1(x))=f_{1,n}\cdot d^{\mathbf{G}}_1(x).$
		It is enough to prove the property for the basis elements. Clearly
		$d^{\mathbf{F}}_1(\varphi_1(E_{\{i,i+1\}}))=f_{1,n}f_{i,i+1}=f_{1,n}\cdot
		d^{\mathbf{G}}_1(E_{\{i,i+1\}})$.  Now
		$d^{\mathbf{F}}_1(\varphi_1(E_1))=d^{\mathbf{F}}_1\left(\sum_{k=1}^{n-1}(b_1)_k e_{\{k,k+1\}} \right)=\sum_{k=1}^{n-1}\frac{Y}{y_ky_{k+1}}f_{k,k+1}.$
		Note that
		$\frac{f_{k,k+1}}{y_ky_{k+1}} =\frac{x_k}{y_k} -\frac{x_{k+1}}{y_{k+1}}$. Now, taking the summation over $k$, we get $d^{\mathbf{F}}_1(\varphi_1(E_1))=f_{1,n}(y_2\cdots y_{n-1})=f_{1,n}\cdot d^{\mathbf{G}}_1(E_1)$. 
		Let $m\geq 2$. Then,  $d^{\mathbf{F}}_1(\varphi_1(E_m))=d^{\mathbf{F}}_1\left(\sum_{k=1}^{n-1}(b_m)_ke_{\{k,k+1\}}\right) =
		\sum_{k=1}^{n-1}(b_m)_kf_{k,k+1}$. It can be seen that
		\begin{eqnarray*}
			\sum_{k=1}^{m-1}(b_m)_kf_{k,k+1} & = & Y\left[\frac{x_2}{y_2}\cdots
			\frac{x_{m-1}}{y_{m-1}} \frac{x_{m+1}}{y_{m+1}}\left(\frac{x_1}{y_1}-\frac{x_m}{y_m}\right) \right], \\
			(b_m)_m f_{m,m+1} & = & Y \left[\frac{x_1}{y_1} \cdots
			\frac{x_{m-1}}{y_{m-1}}\left(	\frac{x_m}{y_m}-\frac{x_{m+1}}{y_{m+1}}\right) \right], \\
			\sum_{k=m+1}^{n-1}(b_m)_kf_{k,k+1} & = & Y\left[\frac{x_2}{y_2}\cdots
			\frac{x_m}{y_m} \left(\frac{x_{m+1}}{y_{m+1}}-
			\frac{x_{n}}{y_{n}}\right) \right].
		\end{eqnarray*}
		Summing up these three terms together, we get 
		\[d^{\mathbf{F}}_1(\varphi_1(E_m))=Y\left[\frac{x_2}{y_2}\cdots
		\frac{x_m}{y_m} \left( 
		\frac{x_{1}}{y_{1}} -\frac{x_n}{y_n}\right) \right]=x_2\cdots
		x_my_{m+1}\cdots y_{n-1}f_{1,n}=f_{1,n}\cdot d^{\mathbf{G}}_1(E_m).\] 
		Therefore, by the mapping cone construction, we get a presentation of
		$J_{C_n}$ as 
		$$ F_2\oplus G_1\longrightarrow F_1\oplus G_0\longrightarrow
		J_{C_n}\longrightarrow 0. $$ 
		Since $F_2\oplus G_1\simeq S^{{n\choose 2} +n-1}$
		and $F_1\oplus G_0\simeq S^n$ whose ranks coincide with the corresponding
		Betti numbers of $J_{C_n}$, we can conclude that this is a minimal presentation. Hence,   the
		first syzygy of $J_{C_n}$ is minimally generated by the images of the
		standard basis elements under the map $\Phi : F_2 \oplus G_1
		\rightarrow F_1 \oplus G_0$, where $\Phi = \begin{bmatrix} d_2^{\bf F}
		& \varphi_1 \\
		0 & -d_1^{\bf G} \end{bmatrix}$. Then,  we have
		\begin{eqnarray*}
			\Phi(e_{\{i,j\},\{k,l\}}) & = & d_2^{\bf F}(e_{\{i,j\},\{k,l\}}) =
			f_{k,l}e_{\{i,j\}} - f_{i,j}e_{\{k,l\}} \text{ for }\{i,j\} \neq
			\{k,l\} \in E(P_n),\\
			\Phi(E_{i,i+1}) & = &  (\varphi_1 - d_1^{\bf G})(E_{i,i+1}) = f_{1,n}e_{\{i,i+1\}} - f_{i,i+1} e_{\{1,n\}}
			\text{ for }i = 1, \ldots, n-1,  \text{ and }\\ 
			\Phi(E_m) & = & \varphi_1(E_m) - d_1^{\bf G}(E_1) = \sum_{k=1}^{n-1}
			(b_m)_ke_{\{k,k+1\}} - (b_m)_ne_{\{1,n\}} \text{ for }i = 1, \ldots, n-1.
		\end{eqnarray*}
		Hence,   the assertion follows.
	\end{proof}

	We now describe a minimal generating set for the
	first syzygy of binomial edge ideals of unicyclic graphs. The syzygy
	structure is slightly different for unicyclic graphs of girth $3$. We first
	deal with that case. 
\begin{theorem}\label{syzygy-unicyclic3}
Let $G$ be a unicyclic graph on $[n]$ of girth $3$. Denote the
vertices of the unique cycle of $G$ by $v_1 < v_2 < v_3$. Let
the standard basis of $S(-2)^n$ be denoted by $\{e_{\{i,j\}} ~: ~ \{i, j\}
\in E(G), i<j \}$. Then,  the first syzygy of $J_G$ is minimally generated by
the elements of the form
\begin{enumerate}
	\item[(a)] $x_{v_1}e_{\{v_2,v_3\}}
	-x_{v_2}e_{\{v_1,v_3\}}+x_{v_3}e_{\{v_1,v_2\}},
	y_{v_1}e_{\{v_2,v_3\}}-y_{v_2}e_{\{v_1,v_3\}}+y_{v_3}e_{\{v_1,v_2\}},$
	\item[(b)] $f_{i,j}e_{\{p,l\}} - f_{p,l}e_{\{i,j\}}, \text{ where } \{\{i,j\},\{p,l\}\}
	\not\subset \{\{v_1,v_2\}, \{v_1,v_3\}, \{v_2,v_3\}\}, \{i,j\} \neq\{p,l\}$ $ \text{ and }
	\{i,j\},\{p,l\} \in E(G)$,  
	
	\item[(c)]  $(-1)^{p_A(j)}f_{k,l}e_{\{i,j\}} + (-1)^{p_A(k)}
	f_{j,l}e_{\{i,k\}} + (-1)^{p_A(l)}f_{j,k}e_{\{i,l\}}, \text{ where }   A = \{i,j,k,l\} \in \mathcal{C}_G$ $ \text{ with center at } i.$
\end{enumerate}
\end{theorem}
\begin{proof}
		We proceed by induction on $n = |V(G)| = |E(G)|$. For $n=3$, $G$ is a
		complete graph i.e., $J_G$ is the ideal generated by the set of all
		$2\times 2$ minor of a $2\times 3$ matrix. Then,  it follows from
		Eagon-Northcott complex that the first syzygy of $J_G$ is minimally
		generated by $$\left\{ x_{v_1}e_{\{v_2,v_3\}}
		-x_{v_2}e_{\{v_1,v_3\}}+x_{v_3}e_{\{v_1,v_2\}},
		y_{v_1}e_{\{v_2,v_3\}}-y_{v_2}e_{\{v_1,v_3\}}+y_{v_3}e_{\{v_1,v_2\}}
		\right\}.$$ 
		
		Now, we assume that $n>3$. From Theorem
		\ref{betti-unicyclic3}, we know that the minimal presentation of $J_G$ is of the form
		\[  S(-4)^{\beta_{2,4}(S/J_G)}\oplus S(-3)^{\beta_{2,3}(S/J_G)} \stackrel{\varphi}{\longrightarrow}
		S(-2)^{n} \stackrel{\psi}{\longrightarrow} J_{G}\longrightarrow
		0, \]
		\[ \text{where} ~ \beta_{2,4}(S/J_G)={n \choose 2}+\sum_{v \in V(G)} {\deg_G(v)
			\choose 3}- \sum_{i=1,2,3}\deg_G(v_i)+3\text{ and }\; \beta_{2,3}(S/J_G)=2.\]
		Let $e=\{u,v\}$ be an edge in $G$ such that $u$ is a pendant vertex of $G$. Since $e$ is a cut edge and $u$ is a pendant vertex of $G$,
		$(G\setminus e)_e= (G \setminus u)_{v} \sqcup \{u\}$. Thus,
		$J_{G\setminus e} :f_e = J_{(G\setminus u)_{v}}$. Since $G\setminus
		e$ is also a unicyclic graph having the unique cycle of girth $3$ and $J_{G \setminus e} =
		J_{G \setminus u}$, by induction we get that the first syzygy of
		$J_{G\setminus e}$ is generated by elements of the form 
		\begin{enumerate}
			\item[(a)] $x_{v_1}e_{\{v_2,v_3\}}
			-x_{v_2}e_{\{v_1,v_3\}}+x_{v_3}e_{\{v_1,v_2\}},
			y_{v_1}e_{\{v_2,v_3\}}-y_{v_2}e_{\{v_1,v_3\}}+y_{v_3}e_{\{v_1,v_2\}}$,
			
			\item[(b)] $f_{i,j}e_{\{p,l\}} - f_{p,l}e_{\{i,j\}}, \text{ where } \{\{i,j\},\{p,l\}\}
			\not\subset \{\{v_1,v_2\}, \{v_1,v_3\}, \{v_2,v_3\}\}, \{i,j\} \neq\{p,l\}$ $ \text{ and }
			\{i,j\},\{p,l\} \in E(G\setminus e)$,
			
			\item[(c)] $(-1)^{p_A(j)}f_{k,l}e_{\{i,j\}} + (-1)^{p_A(k)}
			f_{j,l}e_{\{i,k\}} + (-1)^{p_A(l)}f_{j,k}e_{\{i,l\}}, \text{ where }
			A = \{i,j,k,l\} \in \mathcal{C}_{G \setminus e}$ $\text{ with center at } i.$
		\end{enumerate}
		
		\vskip 2mm
		\noindent
		\textbf{Case-1:} We assume that $v\neq v_i$ for all $1\leq i\leq 3$. 
		Now, we apply the mapping cone construction to the short exact
		sequence (\ref{main-ses}). Let $(\mathbf{G}.,d^{\mathbf{G}}.)$ be a
		minimal free resolution of $[S/(J_{G\setminus e} :f_e)](-2)$. Then
		$G_1 \simeq S^{|E(G)|-1+\binom{\deg_G(v)-1}{2}}.$ Also,     let
		$(\mathbf{F}.,d^{\mathbf{F}}.)$ be a minimal free resolution of
		$S/J_{G\setminus e}$. Then,  $F_1\simeq S^{|E(G)|-1}$ and
		$F_2 \simeq S^{\beta_{2}(S/J_{G\setminus e})}$. By Theorem
		\ref{betti-unicyclic3}, 
		$\beta_2(S/J_{G\setminus e})= 2+
		\beta_{2,4}(S/J_{G\setminus e}),$ where $\beta_{2,4}(S/J_{G\setminus
			e})={n-1 \choose 2}+\sum_{w \in V(G) \setminus v} {\deg_G(w)
			\choose 3}+\binom{\deg_G(v)-1}{3}$. Set $\mathcal{S}_1= \{E_{\{i,j\}}
		:  ~ \{i,j\}\in E(G \setminus u) \}$ and
		$\mathcal{S}_2= \{E_{\{i,j\}} : i,j \in N_G(v)\setminus u \}$.
		Then,  $|\mS_1| = |E(G \setminus e)| = n-1$ and $|\mS_2| = |E( (G\setminus
		e)_e) \setminus E(G \setminus e)| = {\deg_G(v) - 1 \choose 2}$. Let
		$\mathcal{S}_1\cup \mathcal{S}_2$ denote the standard basis of $G_1$
		and set
		$d_1^{\mathbf{G}}(E_{\{i,j\} })=f_{i,j}$ for $E_{\{i,j\}} \in
		\mathcal{S}_1 \cup \mathcal{S}_2$.
		Also, let $\{e_{\{i,j\}} : \{i,j\}\in E(G \setminus u) \}$ be the
		standard basis of $F_1$. By the mapping cone construction, the map
		from $G_0$ to $F_0$ is multiplication by $f_{u,v}$. Define
		$\varphi_1 :  G_1  \rightarrow F_1$ by 
		
		$
		\varphi_1(E_{\{i,j\} }) =  \left\{
		\begin{array}{ll}
		f_{u,v}\cdot e_{\{i,j\} } & \text{ if } E_{\{i,j\} }\in \mathcal{S}_1, \\
		(-1)^{p_A(j)+p_A(u)+1}f_{i,u}e_{\{j,v\} }+(-1)^{p_A(i)+p_A(u)+1}f_{j,u}e_{\{i,v\} }
		& \text{ if } E_{\{i,j\} }\in \mathcal{S}_2. 
		\end{array} \right.
		$
		
		Then,  to prove that $\varphi_1$ is
		a lifting map from $G_1$ to $F_1$ in the mapping cone construction, it
		is enough to show that the corresponding diagram commutes i.e.,
		$d_1^{\mathbf{F}}(\varphi_1(x))=f_{u,v} \cdot d_1^{\mathbf{G}}(x)$ for all
		$x\in G_1$. If $i, j \in N_G(v)\setminus u$, then
		$\{v,u,i,j\}$ is an induced claw with center $v$ and it can be easily seen that
		\[(-1)^{p_A(j)+p_A(u)+1}f_{i,u}f_{j,v}+(-1)^{p_A(i)+p_A(u)+1}f_{j,u}f_{i,v}-f_{i,j}f_{u,v}=0.\] 
		Therefore, it follows that for $E_{\{i,j\}} \in \mathcal{S}_1 \cup
		\mathcal{S}_2$, $d_1^{\mathbf{F}}(\varphi_1(E_{\{i,j\} }))=f_{u,v} \cdot
		d_1^{\mathbf{G}}(E_{\{i,j\} })$.
		Hence,   the mapping cone construction gives a $S$-free presentation
		of $J_G$, which is 
		\begin{eqnarray}\label{map-cone-ses}
		F_2\oplus G_1\longrightarrow F_1\oplus
		G_0\longrightarrow  J_G\longrightarrow 0. 
		\end{eqnarray}
		Since $F_2\oplus
		G_1\simeq S^{\beta_{2}(S/J_G)}$ and $F_1\oplus G_0 \simeq S^n$, the above
		presentation is a minimal one.
		\vskip 2mm \noindent
		\textbf{Case-2:} Let $v=v_i$ for some $1\leq i\leq 3$. Assume that
		$v = v_1$. Then,  $\{v_2,v_3\} \in E( (G\setminus e)_e) \cap
		E(G \setminus e)$. Hence,   $\beta_{1,2}(S/J_{(G\setminus u)_v}) = \rank
		G_1 = (n-1) + {\deg_G(v) - 1 \choose 2} - 1$.
		Also, it follows from Theorem \ref{betti-unicyclic3} that
		\begin{eqnarray*}
			\beta_2(S/J_{G\setminus e}) & = & 2 + 
			{n-1 \choose 2} + \sum_{x \in V(G) \setminus u} {\deg_{G\setminus e}(x) \choose 3}
			- \sum_{i=1}^3 \deg_{G\setminus e}(v_i) + 3. 
		\end{eqnarray*}
		Note that $\deg_G(v_1) = \deg_{G\setminus e}(v_1) + 1$ and
		$\deg_G(x) = \deg_{G \setminus e}(x)$ for all $x \neq u$ and $x \neq
		v$. Substituting these values in the above expression and taking
		summation with $\rank G_1$, we see that $\rank F_2 + \rank G_1 =
		\beta_2(S/J_G)$.
		Let $\mathcal{S}_1 = \{E_{\{i,j\}} :  ~ \{i,j\}\in E(G \setminus u)
		\}$ and $\mathcal{S}_2= \left\{E_{\{i,j\}} : i,j \in N_G(v)\setminus
		u, \{i,j\} \neq \{v_2, v_3\}\right\}.$
		Define $\varphi_1 : G_1 \longrightarrow F_1$ as in \textbf{Case-1} and
		proceeding as in there, it can be proved that the mapping
		cone construction gives a minimal $S$-free presentation of $J_G$ as in
		(\ref{map-cone-ses}). The first syzygy is minimally generated by the
		images of the standard basis under the map $\Phi : F_2 \oplus G_1
		\longrightarrow F_1 \oplus G_0$ which is given by the matrix
		$\begin{bmatrix} d_2^{\bf F} & \varphi_1 \\ 0 & -d_1^{\bf G}
		\end{bmatrix}$. Now, as done in the proof of Theorem
		\ref{syzygy-cycle}, one concludes that the images under $\Phi$ are
		precisely the elements given in the assertion of the theorem.
	\end{proof}
	
	\begin{theorem}\label{syzygy-unicyclic}
		Let $G$ be a unicyclic graph on $[n]$ of girth $m\geq
		4$. Also,     let the vertex set of the unique cycle in $G$ be $\{1,\ldots,m\}$.
		Let $\{e_{\{i,j\}} ~: ~ \{i, j\} \in E(G) \}$ denote the standard basis of $S^{n}$.	
		Then,  the first syzygy of $J_G$ is minimally generated by elements of
		the form
		\begin{enumerate}
			\item[(a)] $f_{i,j}e_{\{k,l\}} - f_{k,l}e_{\{i,j\}}$, where $
			\{i,j\},\{k,l\}\in E(G)$ and $\{i,j\} \neq\{k,l\}$,
			\item[(b)] $(-1)^{p_A(v)}f_{z,w}e_{\{u,v\}} + (-1)^{p_A(z)}f_{v,w}e_{\{u,z\}} + (-1)^{p_A(w)}f_{v,z}e_{\{u,w\}},$
			where
			$A=\{u,v,w,z\} \in \mathcal{C}_G$  with center at $u$,
			\item[(c)] $\sum_{k=1}^{m-1} (b_i)_ke_{\{k,k+1\}}-(b_i)_m
			e_{\{1,m\}}$, where $1 \leq i \leq m-1,$ and $b_i$'s are as
			defined in Theorem \ref{syzygy-cycle}.
		\end{enumerate}
	\end{theorem}
	
\begin{proof}
We prove the assertion by induction on $n-m$. If $n=m$, then $G$ is a
cycle and the result follows from Theorem \ref{syzygy-cycle}. Now, we
assume that $n>m$. From Theorem \ref{betti-unicyclic}, we know that the minimal presentation of $J_G$ is of the form
\[  S^{\beta_2(S/J_G)} \longrightarrow
S^{n} \longrightarrow J_{G}\longrightarrow
0, \]
where
\[ \beta_2(S/J_G) = \left\{
\begin{array}{ll}
\beta_{2,4}(S/J_G) & \text{ if } m = 4 \\
\beta_{2,4}(S/J_G) + \beta_{2,m}(S/J_G) & \text{ if } m > 4,
\text{ and }
\end{array} \right. \]
\[
\beta_{2,4}(S/J_G)= \left\{
\begin{array}{ll}
{n \choose 2}+\sum_{v \in V(G)} {\deg_G(v)
	\choose 3}+ 3 & \text{ if } m = 4 \\
{n \choose 2}+\sum_{v \in V(G)} {\deg_G(v)
	\choose 3} & \text{ if } m > 4
\end{array} \right. \text{ and }
\beta_{2,m}(S/J_G)=m-1.\]
Let $e=\{u,v\}$ be an edge in $G$ such that $u$ is a pendant vertex
of $G$. Since $e$ is a cut edge and $u$ is a pendant vertex of $G$,
$(G\setminus e)_e= (G \setminus u)_{v} \sqcup \{u\}$. Thus,
$J_{G\setminus e} :f_e = J_{(G\setminus u)_{v}}$. Since $G\setminus
e$ is also a unicyclic graph having the unique cycle $C_m$ and $J_{G \setminus e} =
J_{G \setminus u}$, by induction we get a minimal generating set
of the first syzygy of $J_{G\setminus e}$ as 

\noindent
\begin{enumerate}
	\item[(a)] $f_{i,j}e_{\{k,l\}} - f_{k,l}e_{\{i,j\}}$, where $\{i,j\},\{k,l\}\in E(G \setminus e)$ and $\{i,j\} \neq\{k,l\}$, 
	\item[(b)] $(-1)^{p_A(j)}f_{k,l}e_{\{i,j\}} + (-1)^{p_A(k)}f_{j,l}e_{\{i,k\}} + (-1)^{p_A(l)}f_{j,k}e_{\{i,l\}},$
	where
	$A=\{i,j,k,l\} \in \mathcal{C}_{G\setminus e}$  with center at $i$,
	\item[(c)] $\sum_{k=1}^{m-1} (b_i)_ke_{\{k,k+1\}}-(b_i)_m
	e_{\{1,m\}}$, where $1 \leq i \leq m-1.$
\end{enumerate}

Now, we apply the mapping cone construction to the short exact sequence
(\ref{main-ses}). Let $(\mathbf{G}.,d^{\mathbf{G}}.)$ and $(\mathbf{F}.,d^{\mathbf{F}}.)$ be minimal
free resolutions of $[S/(J_{G\setminus e} :f_e)](-2)$ and $S/J_{G\setminus e}$
respectively. Then,  $G_1\simeq S^{n-1+\binom{\deg_G(v)-1}{2}}, F_1 \simeq S^{n-1}$ and
$F_2 \simeq S^{\beta_{2}(S/J_{G\setminus e})}$. 

Denote the standard basis of $G_1$ by $\mathcal{S}_1\cup
\mathcal{S}_2$, where
$\mathcal{S}_1= \{E_{\{i,j\}} : \{i,j\}\in E(G \setminus e) \}$ and $\mathcal{S}_2= \{E_{\{k,l\}} : k,l \in
N_G(v)\setminus u \}$. Note that $|\mS_1| = n-1$ and $|\mS_2| = {\deg_G(v)
	- 1 \choose 2}$.  Set $d_1^{\mathbf{G}}(E_{\{i,j\}
})=f_{i,j}$ for a basis element $E_{\{i,j\} }$. Also,     let
$\{e_{\{i,j\}} : \{i,j\}\in E(G \setminus e) \}$
be the standard basis of $F_1$. By the mapping cone construction, the
map from $G_0$ to $F_0$ is given by the multiplication by $f_e$. Now, we
define $\varphi_1$ from $G_1$ to $F_1$ by $\varphi_1(E_{\{i,j\}
})=f_e\cdot e_{\{i,j\} }$ for $E_{\{i,j\} }\in \mathcal{S} _1$ and
$\varphi_1(E_{\{k,l\} })=(-1)^{p_A(k)+p_A(u)+1}f_{u,l}e_{\{v,k\}
}+(-1)^{p_A(l)+p_A(u)+1}f_{u,k}e_{\{v,l\} }$ for
$E_{\{k,l\} }\in \mathcal{S}_2$. We need to prove that 
$d_1^{\mathbf{F}}(\varphi_1(x))=f_e \cdot d_1^{\mathbf{G}}(x)$ for
any element $x\in G_1$. For a claw $\{v,u,k,l\}$ with center at $v$,
we have the relation $(-1)^{p_A(k)+p_A(u)+1}f_{u,l}f_{v,k}
+(-1)^{p_A(l)+p_A(u)+1}f_{u,k}f_{v,l} = f_{k,l}f_{u,v}.$
This yields us the equality
$d_1^{\mathbf{F}}(\varphi_1(E_{\{i,j\} }))=f_{u,v} \cdot
d_1^{\mathbf{G}}(E_{\{i,j\} })$ for a basis $E_{\{i,j\} }$ of $G_1$.
So the mapping cone construction gives us a $S$-free presentation of
$J_G$ as
$$ F_2\oplus G_1\longrightarrow F_1\oplus G_0\longrightarrow F_0\longrightarrow J_G\longrightarrow 0.$$
Since $F_2\oplus
G_1 \simeq S^{\beta_{2}(S/J_G)}$ and $F_1\oplus G_0 \simeq S^n$, this is a minimal
free presentation. Hence,   the first syzygy of $J_G$ is minimally
generated by the images of basis elements under the map $\Phi: F_2
\oplus G_1 \longrightarrow F_1 \oplus G_0$. Now, the assertion can be
proved just as done in the proof of Theorem \ref{syzygy-cycle}.
\end{proof}

If $e = \{u, v\}$ is a cut-edge in $G$ such that both $u$ and $v$ are
simplicial vertices, then the mapping cone construction on the exact
sequence (\ref{main-ses}) gives a minimal free resolution of $S/J_G$, \cite[Proposition 3.2]{KM-pure}.
However, this is not a necessary condition as we see below.
\begin{proposition}  
Let $n \geq 3$. Then, the minimal free resolution of $S/J_{K_{1,n}}$ is given by the mapping
cone of $S/J_{K_{1,n-1}}$ and $S/J_{K_n}(-2)$. 
\end{proposition}
\begin{proof}
Let $V(K_{1,n}) = \{1,\ldots, n,n+1\}$ with $E(K_{1,n}) = \{\{i,n+1\}: 1
\leq i \leq n \}$. For $G = K_{1,n}$ and $e=\{n,n+1\}$, 
note that $J_{G\setminus e} = J_{K_{1,n-1}}$ and $J_{K_{1,n-1}} :f_e=J_{K_n}$.
Since $K_{1,n}$ is a tree, it follows from \cite[Theorem 1.1]{her1}
that $\pd(S/J_{K_{1,n}})=n$. Also, by \cite[Corollary 2.3]{KM12},
$\beta_{i,i+1}(S/J_{K_{1,n}})=0$ for $2 \leq i \leq n$. 
Since $\reg(S/J_{K_{1,n}})=2$, \cite{SZ14}, and $\reg(S/J_{K_n})=1$,
\cite{KM12},
$\beta_{i,i+j}(S/J_{K_{1,n}})=0$ for $j \neq 2$ and
$\beta_{i,i+j}(S/J_{K_n})=0$ for $j \neq 1$.
Corresponding to (\ref{main-ses}), we have the long exact sequence for
all $j \geq 1$,
\begin{equation}\nonumber
\cdots	\rightarrow
\Tor_{i,i+j}^S\left(\frac{S}{J_{K_{1,n-1}}},\KK\right)\rightarrow
\Tor_{i,i+j}^S\left(\frac{S}{J_{K_{1,n}}},\KK\right) \rightarrow
\Tor_{i-1,i+j-2}^S\left(\frac{S}{J_{K_n}},\KK\right)\rightarrow\cdots.
\end{equation}
Hence,
$\beta_{i,j}(S/J_{K_{1,n}})=\beta_{i,j}(S/J_{K_{1,n-1}})+\beta_{i-1,j-2}(S/J_{K_n})$.
If $\mathbf{G_\cdot}$ denotes a minimal free resolution of $S/J_{K_n}(-2)$ and
$\mathbf{F_\cdot}$ denotes a minimal free resolution of $S/J_{K_{1,n-1}}$,
then the above equality implies that $\beta_{i}(S/J_{K_{1,n}}) =
\rank F_i + \rank G_{i-1}$.  Hence, the mapping cone gives a minimal
free resolution of $S/J_{K_{1,n}}$.
\end{proof}
	
\section{Rees Algebra}
Let $G$ be a graph on $[n]$ and $J_{G}$ be its binomial edge ideal.
Let
$R=S[T_{\{i,j\}} ~ : ~\{i,j\} \in E(G) \text{ with } i<j]$.  
Let $\delta : R \to S[t]$ be the $S$-algebra homomorphism given by
$\delta(T_{\{i,j\}}) = f_{i,j}t$. Then,   Im$(\delta) = \R(J_G)$ and $\ker (\delta)$ is
called the defining ideal of $\R(J_G)$. 
We first characterize graphs whose binomial edge ideals are 
almost complete intersection. We begin by proving couple of simple
lemmas which are useful for our main results.
	
\begin{lemma}\label{radical-lemma}
Let $I$ be a radical ideal in a Noetherian commutative ring $A$. Then,
for any $f \in A$ and $ n \geq 2$, $I:f =I:f^n$.
\end{lemma}
\begin{proof}
Let $f \in A$ be an element. Observe that for any $ n\geq 2$, $I:f
\subset I:f^n$. Let $g \in I:f^n$. Then, $gf^n \in I$ which implies
that $g^nf^n \in I$. Therefore, $gf \in \sqrt{I} = I$. Hence, $g \in I:f$. 
\end{proof}

\begin{lemma}\label{aci-lemma}
If $I \subseteq A = \K[t_1, \ldots, t_n]$ is a homogeneous ideal such
that $I = J + (a)$, where $J$ is generated by a homogeneous
regular sequence, $a$ is a homogeneous element and $J : a = J :
a^2$, then $I$ is either a complete intersection or an almost
complete intersection.
\end{lemma}

\begin{proof}
The proof for Theorem 4.7(ii) in \cite{HMV89} is for the local case for
the same statement, but it can be easily seen that it goes through for
homogeneous ideals in $A$.
\end{proof}

We first characterize the trees whose binomial edge ideals are
almost complete intersections.
\begin{theorem}\label{tree-aci}
If $G$ is a tree which is not a path, then $J_G$ is an almost complete
intersection ideal if and only if $G$ is obtained by adding an edge
between two vertices of two paths. 
\end{theorem}
\begin{proof}
Suppose $G$ is obtained by adding an edge $e$ between paths $P_{n_1}$
and $P_{n_2}$.
Then,  $J_{G \setminus e}$ is a complete intersection ideal and $J_G =
J_{G \setminus e}+f_eS$. By Theorem \ref{tech-thm}(a), and
Lemma \ref{radical-lemma}, we get
$J_{G\setminus e}:f_e^2 =J_{G \setminus e}:f_e$. Therefore, it
follows from Lemma \ref{aci-lemma} that
$J_G$ is an almost complete intersection.  

Now, assume that $G$ is not a graph obtained by adding an edge between
two paths.  Therefore, either there exists a vertex $v$ such that $\deg_G(v)\geq 4$ or there exist $z,w \in
V(G)$ such that $\deg_G(z)\geq 3, \; \deg_G(w) \geq 3$ and $\{z,w\}
\notin E(G)$. Let $T=\{v\}$ in the first case and $T =\{z,w\}$ in
the second case. By Theorem \ref{tech-thm}, $\h (P_T(G)) = n - c_T +
| T |$. Since $z$ and $w$ are of degrees at least $3$, $\{z,w\} \notin E(G)$ and $G$ is a
tree,  $c_T \geq 5$. Hence, $\h
(P_T(G)) \leq n-3$. Now, if $T=\{v\}$, then $c_T\geq 4$ so that $\h
(P_T(G)) \leq n-3$. Note that in both cases $T$ has the cut point
property so that $P_T(G)$ is a
minimal prime, by Theorem \ref{tech-thm}. Thus,    $\h(J_G) \leq n-3$. Since
$\mu(J_G)=n-1$, $\mu(J_G)>\h(J_G)+1$. Hence,   $J_G$ is not an almost
complete intersection ideal.
\end{proof}	
Now, we have characterized the almost complete intersection trees,
we move on to graphs containing cycles.
\begin{theorem}\label{graph-aci}
Let $G$ be a connected graph on $[n]$ which is not a tree. Then,  $J_G$
is an almost complete intersection ideal if and only if $G$ is
obtained by adding an edge between two vertices of a path 
or by attaching a path to each vertex of  $C_3$. 
\end{theorem}
\begin{proof}
First assume that $J_G$ is an almost complete intersection ideal. Therefore, $\mu(J_G)=\h(J_G)+1$. Since
$\h(J_G) \leq n-1$, it follows that $\mu(J_G)\leq
n$. Since $G$ is not a tree, we have $\mu(J_G)=n$. Therefore, $G$ is
a unicyclic graph and $\h(J_G)=n-1$. Let $u$ be a vertex which does
not belong to the unique cycle in $G$. If
$\deg_G(u) \geq 3$, then for $T=\{u\}$, by Theorem
\ref{tech-thm}(d), $P_T(G)$ is a
minimal prime of $J_G$ of height $\leq n-2$ which contradicts the fact
that $\h(J_G)
=n-1.$ Hence, $\deg_G(u) \leq
2$. Now, we claim that $\deg_G(u)\leq 3$, for every $u$
belonging to
vertex set of the unique cycle in $G$. If  $\deg_G(u) \geq 4$ for
such a vertex $u$, then $G \setminus u$ has at least three components
so that for $T=\{u\}$,
$P_T(G)$ is a minimal prime of $J_G$ of height $\leq n-2$ which is a
contradiction. Hence,   $\deg_G(u) \leq 3$. If the girth of $G$ is $3$,
then clearly it belongs to one of the categories described in the
theorem. We
now assume that girth of $G$ is $\geq 4$. Suppose $u,v$ 
be two vertices of the unique cycle in $G$ with $\deg_G(u) =3$ and $\deg_G(v)=3$.
If $\{u,v\} \notin E(G)$, then for $T=\{u,v\}$,
$P_T(G)$ is a minimal prime of $J_G$ of height $\leq n-2$ which is
again a contradiction. Therefore, $\{u, v\} \in E(G)$. Thus, the
number of vertices of the cycle having degree three is at most $2$ and if
two vertices of the cycle have degree three, then they are adjacent.
Therefore, $G$ is obtained by adding an edge between two vertices of a
path. 
		
		Now assume that $G$ is a graph obtained by adding an edge between two
		vertices, say $u$ and $v$, of a path. Let $e = \{u,v\}$.
		Observe that $J_{G \setminus
			e}$ is a complete intersection ideal. By Theorem
			\ref{tech-thm}(a) and Lemma \ref{radical-lemma}, $J_{G\setminus
			e}:f_e^2 =J_{G \setminus e}:f_e$. Thus, it follows from
		Lemma \ref{aci-lemma} that $J_G$ is an almost complete
		intersection ideal.
		
		Now, suppose $G$ is a graph obtained by adding a path to each of the
		vertices of a $C_3$.
		Then, by \cite[Theorem 1.1]{her1}, $S/J_G$ is Cohen-Macaulay of
		dimension $n+1$.
		Therefore, $\h(J_G) =n-1 = \mu(J_G) - 1$. Now, we have to prove that
		if $\mathfrak{p}$ is a minimal prime of $J_G$, then
		$(J_G)_{\mathfrak{p}}$ is a complete intersection ideal of
		$S_{\mathfrak{p}}$, i.e.
		$\mu((J_G)_{\mathfrak{p}})=\h((J_G)_{\mathfrak{p}})=n-1$. Let
		$\mathfrak{p}$ be a minimal prime of $J_G$. It follows from
		\cite[Corollary 3.9]{HH1} that there exists  $T \subset [n]$ having
		cut point property such that
		$\mathfrak{p}=P_T(G)$. By Theorem \ref{betti-unicyclic3},
		the minimal presentation of $J_G$ is
		\[  S(-4)^{\beta_{2,4}(S/J_G)}\oplus S(-3)^{\beta_{2,3}(S/J_G)} \stackrel{\varphi}{\longrightarrow}
		S(-2)^{n}\longrightarrow J_{G}\longrightarrow 0.\] Moreover, 
		the linear relations given in Theorem \ref{syzygy-unicyclic3}(a) show
		that $(x_{v_1},y_{v_1},x_{v_2},y_{v_2},x_{v_3},y_{v_3}) \subset
		I_1(\varphi),$ the ideal generated by the entries of the matrix of
		$\varphi$. 
		Now, if $I_1(\varphi)\subset \mathfrak{p}$, then
		$(x_{v_1},y_{v_1},x_{v_2},y_{v_2},x_{v_3},y_{v_3}) \subset
		\mathfrak{p}$. Thus,    $\{v_1,v_2,v_3\}\subset T$, which is a
		contradiction to the fact that $T$ has the cut point property. Therefore,
		$I_1(\varphi) \not\subset \mathfrak{p}$, and hence, by \cite[Lemma
		1.4.8]{bh}, $\mu((J_G)_{\mathfrak{p}})\leq n-1$.  If
		$\mu((J_G)_{\mathfrak{p}})< n-1$, then by \cite[Theorem 13.5]{hm},
		$\h(\mathfrak{p})<n-1$, which is a contradiction. Thus,
		$\mu((J_G)_{\mathfrak{p}})= n-1$.  Hence, $J_G$ is an almost complete
		intersection ideal.
	\end{proof}
	
	Below, we give representatives of four different types of graphs
	whose binomial edge ideals are almost complete intersection ideals.
	
	\vskip 2mm
	
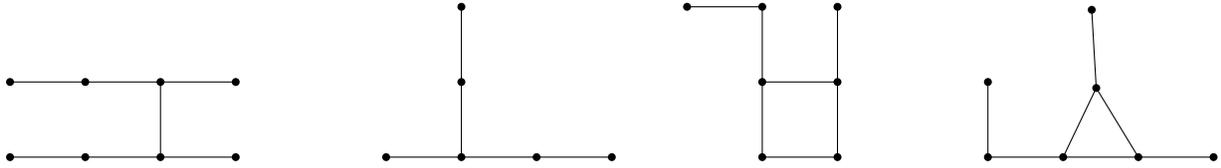
\noindent
	\begin{tikzpicture}[scale=1]
	\draw (-2,-3)-- (-1,-3);
	\draw (-1,-3)-- (0,-3);
	\draw (0,-3)-- (1,-3);
	\draw (0,-3)-- (0,-4);
	\draw (0,-4)-- (-1,-4);
	\draw (-1,-4)-- (-2,-4);
	\draw (0,-4)-- (1,-4);
	\draw (3,-4)-- (4,-4);
	\draw (4,-4)-- (5,-4);
	\draw (4,-4)-- (4,-3);
	\draw (4,-3)-- (4,-2);
	\draw (5,-4)-- (6,-4);
	\draw (8,-3)-- (8,-4);
	\draw (8,-4)-- (9,-4);
	\draw (8,-3)-- (9,-3);
	\draw (9,-3)-- (9,-4);
	\draw (9,-3)-- (9,-2);
	\draw (8,-3)-- (8,-2);
	\draw (8,-2)-- (7,-2);
	\draw (12,-4)-- (11,-4);
	\draw (12,-4)-- (12.44,-3.08);
	\draw (12.44,-3.08)-- (13,-4);
	\draw (12,-4)-- (13,-4);
	\draw (13,-4)-- (14,-4);
	\draw (12.44,-3.08)-- (12.38,-2.04);
	\draw (11,-4)-- (11,-3);
	\begin{scriptsize}
	\fill  (-2,-3) circle (1.5pt);
	\fill  (-1,-3) circle (1.5pt);
	\fill  (0,-3) circle (1.5pt);
	\fill  (1,-3) circle (1.5pt);
	\fill  (0,-4) circle (1.5pt);
	\fill  (-1,-4) circle (1.5pt);
	\fill (-2,-4) circle (1.5pt);
	\fill  (1,-4) circle (1.5pt);
	\fill  (3,-4) circle (1.5pt);
	\fill  (4,-4) circle (1.5pt);
	\fill  (5,-4) circle (1.5pt);
	\fill  (4,-3) circle (1.5pt);
	\fill  (4,-2) circle (1.5pt);
	\fill  (6,-4) circle (1.5pt);
	\fill  (8,-3) circle (1.5pt);
	\fill  (8,-4) circle (1.5pt);
	\fill  (9,-4) circle (1.5pt);
	\fill  (9,-3) circle (1.5pt);
	\fill  (9,-2) circle (1.5pt);
	\fill  (8,-2) circle (1.5pt);
	\fill  (7,-2) circle (1.5pt);
	\fill  (12,-4) circle (1.5pt);
	\fill  (11,-4) circle (1.5pt);
	\fill  (12.44,-3.08) circle (1.5pt);
	\fill  (13,-4) circle (1.5pt);
	\fill (14,-4) circle (1.5pt);
	\fill  (12.38,-2.04) circle (1.5pt);
	\fill  (11,-3) circle (1.5pt);
	\end{scriptsize}
	\end{tikzpicture}
	
	\vskip 2mm
	We now study the Rees algebra of almost complete intersection binomial
	edge ideals. We prove that they are Cohen-Macaulay and we also obtain
	the defining ideals of these Rees algebras. We first recall a result
	which characterizes the Cohen-Macaulayness of the Rees algebra and the
	associated graded ring.
	\begin{theorem}\cite[Corollary 1.8]{Herr}\label{aci-cmrees}
		Let $A$ be a Cohen-Macaulay local (graded) ring and $I \subset A$ be a
		(homogeneous) almost complete intersection ideal in $A$. Then,  
		\begin{enumerate}[(a)]
			\item $\gr_A(I)$ is Cohen-Macaulay if and only if
			$\depth(A/I) \geq \dim(A/I) -1.$
			\item $\R(I)$ is Cohen-Macaulay if and only if $\h(I) > 0$ and
			$\gr_A(I)$ is Cohen-Macaulay.
		\end{enumerate}
	\end{theorem}
	Therefore, in our situation, to prove that $\R(J_G)$ is
	Cohen-Macaulay, it is enough to prove that $\depth(S/J_G) \geq
	\dim(S/J_G) - 1$.

	\subsection{Discussion}\label{disc}
	Suppose $G$ is a unicyclic graph such that $J_G$ is almost complete
	intersection. We may assume that $G$ is not a cycle.
	If girth of $G$ is $3$, then by Theorem \ref{graph-aci}
	and \cite[Theorem 1.1]{her1}, $S/J_G$ is Cohen-Macaulay. Thus,
	$\gr_{S}(J_G)$ is Cohen-Macaulay, and hence, so is $\R(J_G)$. Now, we
	assume that girth of $G$ is at least $4$ and $n \geq 5$.

\noindent
\begin{minipage}\linewidth
	\begin{minipage}{0.57\linewidth}
		Let $G_1$ and $G_2$ denote 
		graphs on the vertex set $[n]$ with edge sets given by $E(G_1)
		=\{\{1,2\},\{2,3\},\ldots,\{n-2,n-1\}, \{n-1,n\},\{2,n-1\}\}$ and
		$E(G_2)=\{\{1,2\},\{2,3\},\ldots, \{n-1,n\},\{2,n\}\}$. If $G$ is a
		unicyclic graph on $[k]$, $k \geq 5$, which is not a cycle and having
		an almost complete intersection binomial edge ideal, then by
		Theorem \ref{graph-aci}, $G$ is obtained by attaching a path to each
		of the pendant vertices of $G_1$ or $G_2$. 
	\end{minipage}
	\begin{minipage}{0.40\linewidth}
		
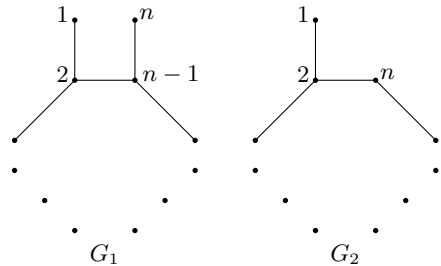
\hspace*{3mm}\begin{tikzpicture}[scale=0.8]
		\draw  (2,4)-- (2,3);
		\draw  (3,4)-- (3,3);
		\draw  (3,3)-- (4,2);
		\draw  (2,3)-- (1,2);
		\draw  (5,2)-- (6,3);
		\draw  (6,3)-- (6,4);
		\draw  (2,3)-- (3,3);
		\draw  (6,3)-- (7,3);
		\draw  (7,3)-- (8,2);
		\begin{scriptsize}
		\draw [fill=black] (2,4) circle (1pt);
		\draw [fill=black] (2,3) circle (1pt);
		\draw [fill=black] (3,4) circle (1pt);
		\draw [fill=black] (3,3) circle (1pt);
		\draw [fill=black] (4,2) circle (1pt);
		\draw [fill=black] (1,2) circle (1pt);
		\draw [fill=black] (5,2) circle (1pt);
		\draw [fill=black] (6,3) circle (1pt);
		\draw [fill=black] (6,4) circle (1pt);
		\draw [fill=black] (7,3) circle (1pt);
		\draw [fill=black] (8,2) circle (1pt);
		\draw [fill=black] (1,1.5) circle (1pt);
		\draw [fill=black] (1.5,1) circle (1pt);
		\draw [fill=black] (2,0.5) circle (1pt);
		\draw [fill=black] (3,0.5) circle (1pt);
		\draw [fill=black] (3.5,1) circle (1pt);
		\draw [fill=black] (4,1.5) circle (1pt);
		\draw [fill=black] (5,1.5) circle (1pt);
		\draw [fill=black] (5.5,1) circle (1pt);
		\draw [fill=black] (6,0.5) circle (1pt);
		\draw [fill=black] (7,0.5) circle (1pt);
		\draw [fill=black] (7.5,1) circle (1pt);
		\draw [fill=black] (8,1.5) circle (1pt);
		\draw (7.2, 3.1) node {$n$};
		\draw (5.8, 3.1) node {$2$};
		\draw (5.8, 4.1) node {$1$};
		\draw (3.2, 4.1) node {$n$};
		\draw (1.8, 4.1) node {$1$};
		\draw (1.8, 3.1) node {$2$};
		\draw (3.6, 3.1) node {$n-1$};
		\draw (2.5, 0.1) node {$G_1$};
		\draw (6.5, 0.1) node {$G_2$};
		\end{scriptsize}
		\end{tikzpicture}
	\end{minipage}
\end{minipage}

Let $G$ denote the
graph obtained by identifying the vertex $1$ of $G_i$ and a pendant
vertex of $P_m$. Then,  by \cite[Theorem 2.7]{RR14},
$\depth(S/J_G) = \depth(S_i/J_{G_i}) + \depth(S_P/J_{P_m})
-2$, where $S_i$ denotes the polynomial ring corresponding to the
graph $G_i$ and $S_P$ denotes the polynomial ring corresponding to the
graph $P_m$.  Since $J_{P_m}$ is generated by a
regular sequence of length $m-1$, $\depth(S_P/J_{P_m}) = m+1$. Also
$\dim(S/J_G) = n+m$.
Therefore, to prove that $\depth(S/J_G) \geq n+m-1$, it is
enough to prove that $\depth(S_i/J_{G_i}) \geq n$. Similarly, if $G$ is
obtained by attaching a path to each of the pendant vertices of $G_1$,
then to prove $\depth(S/J_G) \geq \dim (S/J_G) - 1$, it is enough to
prove that $\depth(S_1/J_{G_1}) \geq \dim(S_1/J_{G_1}) - 1$. We now
proceed to prove this.

Let $G$ be a graph on $[n]$ with binomial edge ideal $J_G \subset
S = \K[x_1,\ldots, x_n,y_1,\ldots, y_n]$. We
consider $S$ with lexicographical order induced by $x_1 > \cdots > x_n
> y_1 > \cdots > y_n$. It follows from \cite[Theorem
2.1]{HH1} that $\ini_{<}(J_G)$ is a squarefree monomial ideal so that
by
\cite[Corollary 2.7]{CV18}, we get $\depth(S/J_G) =
\depth(S/\ini_<(J_G))$. Hence,   to compute $\depth(S/J_G)$, we compute the
depth of $S/\ini_<(J_G)$.

Now, consider the graphs $G_1$ and $G_2$ as defined above. It follows from the labeling of the
vertices of $G_1$ that the admissible 
paths in $G_1$ are the edges and the paths of the form $i,
i-1,\ldots,3,2,n-1,n-2,\ldots,j$ with  $2\leq j-i\leq n-4$,
\cite[Section 2]{HH1}.
Similarly the admissible paths in $G_2$ are the edges and the paths of
the form $i,i-1,\ldots,3,2,n,n-1,\ldots,j$ with $2 \leq j-i \leq
n-3$. Consequently, the corresponding initial ideals are given
by  
\[
\ini_<(J_{G_1})=
\Big(\{x_1y_2,\ldots,x_{n-1}y_n,x_2y_{n-1},
x_ix_{j+1}\cdots x_{n-1}y_2\cdots y_{i-1}y_j : 2\leq j-i\leq n-4\}
\Big) \text{ and } \]
\[\ini_<(J_{G_2})= \Big(\{x_1y_2,\ldots, x_{n-1}y_n,x_2y_{n},
x_ix_{j+1}\cdots x_{n}y_2\cdots y_{i-1}y_j : 2\leq j-i\leq n-3\}\Big).\]
We denote these  monomials of degree $\geq 3$ by 
$v_1,\ldots, v_p$. We order these monomials such that 
$i<j$ if either $\deg v_i<\deg v_j$ or $\deg v_i=\deg v_j$
and $v_i>_{lex}v_j$.
Set $J=(x_1y_2,\ldots,x_{n-1}y_n)$, $I_0(G_1)=J+(x_2y_{n-1})$,
$I_0(G_2)=J+(x_2y_{n})$ and, for $1\leq k\leq p,$
$I_k(G_i)=I_{k-1}(G_i)+(v_k)$ for $i = 1, 2$. Then
$I_p(G_i)=\ini_<(J_{G_i})$ for $i = 1, 2$. We now compute the
projective dimension, equivalently depth, of these ideals.
	
	\begin{lemma}\label{initial-ideal}
		For $0\leq k\leq p$ and $i = 1, 2$, $\pd (S/I_k(G_i))\leq n$.
	\end{lemma}
	\begin{proof}
		We prove the assertion by induction on $k$. If $k=0$, then consider the following exact sequences:
		\[0\longrightarrow
		\frac{S}{J:(x_2y_{n-1})}(-2)\stackrel{\cdot x_2y_{n-1}}{\longrightarrow}\frac{S}{J}\longrightarrow
		\frac{S}{I_0(G_1)}\longrightarrow 0\] and
		\[0\longrightarrow
		\frac{S}{J:(x_2y_n)}(-2)\stackrel{\cdot x_2y_{n}}{\longrightarrow}\frac{S}{J}\longrightarrow
		\frac{S}{I_0(G_2)}\longrightarrow 0.\]
		Note that $J$ is generated by a regular sequence of length $n-1$.
		Moreover 
		\begin{eqnarray*}
			J :x_2y_{n-1} & = &
			(x_1y_2,y_3,x_3y_4,\ldots,x_{n-3}y_{n-2},x_{n-2},x_{n-1}y_n) \text{
				and } \\
			J :x_2y_{n} & = &
			(x_1y_2,y_3,x_3y_4,\ldots,x_{n-3}y_{n-2},x_{n-2}y_{n-1},x_{n-1})
		\end{eqnarray*}
		which are generated by regular sequences of length $n-1$. Therefore
		\[\pd(S/J) = \pd(S/(J:x_2y_{n-1})) =
		\pd(S/(J:x_2y_{n}))=n-1.\] Hence, it follows from the long exact
		sequence of Tor that $\pd(S/I_0(G_i))\leq n$ for $i = 1,2$. 
		Now, assume that $k>0$ and $\pd(S/I_{k-1}(G_i))\leq n$ for $i = 1, 2$.
		For $i = 1, 2$, consider  the short exact sequences  
		\begin{eqnarray}\label{lemma-ses}
		0\longrightarrow \frac{S}{I_{k-1}(G_i):(v_k)}(-\deg
		v_k)\stackrel{\cdot v_k}{\longrightarrow} \frac{S}{I_{k-1}(G_i)}\longrightarrow
		\frac{S}{I_k(G_i)} \longrightarrow 0.
		\end{eqnarray}
		We first prove the assertion for $G_1$.  It can be seen that the
		monomials $v_k$'s are of the form 
		\[v_k= \left\{
		\begin{array}{ll}
		x_2x_{j+1}\cdots x_{n-1}y_j & \text{ for } 4 \leq j \leq n-2, \\
		x_iy_2\cdots y_{i-1}y_{n-1}  & \text{ for } 3 \leq i \leq n-3, \\
		x_ix_{j+1}\cdots x_{n-1}y_2\cdots y_{i-1}y_j & \text{ for } 3 \leq
		i; j \leq n-2 \text{ and } 2 \leq j-i.
		\end{array}\right.\]
		If $v_i = x_2x_{j+1}\cdots x_{n-1}y_j$ for some $4 \leq j \leq n-2$,
		then 
		\begin{eqnarray*}
			I_{k-1}(G_1) : v_k & = & (I_0(G_1) : v_k) + (v_1, \ldots, v_{k-1}) :
			v_k \\
			& = & (x_1y_2,x_3y_4,\dots,x_{j-2}y_{j-1},x_jy_{j+1},
			x_{j-1},y_3,y_{j+2},\dots,y_n) + (v_1, \ldots, v_{k-1}) : v_k. 
		\end{eqnarray*}
		It can be seen that $(v_1, \ldots, v_{k-1}) : v_k
		\subseteq (I_0(G_1) : v_k) + (y_{j+1})$ and $y_{j+1}v_k \in (v_1,
		\ldots, v_{k-1})$. Hence,   
		\[I_{k-1}(G_1) : v_k = (x_1y_2,x_3y_4,\dots,x_{j-2}y_{j-1}, 
		x_{j-1},y_3,y_{j+1},\dots,y_n).
		\]
		This is a regular sequence of length $n-1$. The
		proof that $I_{k-1}(G_1) : v_k$ is generated by a regular sequence of
		length $n-1$ if $v_k$ is of the other two types is similar.
		Therefore
		$\pd(S/(I_{k-1}(G_1) : v_k)) = n-1$. Hence,   it follows from the short
		exact sequence (\ref{lemma-ses}) that $\pd(S/I_k(G_1)) \leq n$. 
		
		In a similar manner, using the short exact sequence (\ref{lemma-ses})
		and the colon ideal, one can prove that $\pd(S/I_k(G_2)) \leq n$.
	\end{proof}

We now show that the associated graded ring and the
Rees algebra of  almost complete intersections
binomial edge ideals  are Cohen-Macaulay.
	
	\begin{theorem}\label{uni-cmrees}
		If  $G$ is a  graph such that $J_G$ is an almost complete intersection ideal, then 		$\gr_S(J_G)$ and $\R(J_G)$ are Cohen-Macaulay.
	\end{theorem}
	\begin{proof} Suppose $J_G$ is an almost complete intersection
	  ideal. By Theorem \ref{aci-cmrees}(b), it is enough
		to prove that $\gr_S(J_G)$ is Cohen-Macaulay, if one wants to
		prove that $\R(J_G)$ is Cohen-Macaulay. Now, $\gr_S(J_G)$ is
		Cohen-Macaulay if $\depth(S/J_G) \geq \dim(S/J_G) 
		- 1$, by Theorem \ref{aci-cmrees}(a). If $G$ is a tree, then it follows from \cite[Theorem 1.1]{her1} and Theorem \ref{tree-aci}
		that $\depth (S/J_G)=n+1 = \dim(S/J_G)-1$. If $G = C_n$, then it follows from \cite[Theorem
		4.5]{Za12} that  $\depth(S/J_{C_n}) = \dim(S/J_{C_n})-1$. Now, we assume that $G$ is a unicyclic graph other than cycle.
		  It follows from Discussion \ref{disc} that  it is enough to prove that
		$\depth(S_i/J_{G_i}) \geq n$ for $i = 1,2$.
		From \cite[Corollary 2.7]{CV18}, we get $\depth(S_i/J_{G_i}) =
		\depth(S_i/\ini_>(J_{G_i}))$. It follows from Lemma
		\ref{initial-ideal} that 
		$\depth(S_i/\ini_>(J_{G_i})) = \depth(S_i/I_p(G_i)) \geq n$. This completes the proof.
	\end{proof}
	
	We now study binomial edge ideals which are of linear type.  
	Since complete intersections are of linear type, binomial edge
	ideals of paths are of linear type.
	%
	Now, we show that the
	$J_{K_{1,n}}$ is of linear type. For this purpose,
	recall the definition of $d$-sequence. 
	\begin{definition}
		Let $A$ be a commutative ring. Set $d_0 = 0$. A sequence of
		elements $d_1 ,\dots ,d_n$ is said to be a $d$-sequence if 
		$(d_0,d_1 ,\dots ,d_i)  : d_{i+1}d_j = (d_0,d_1 ,\dots ,d_i) : d_j$ for all $0 \leq i\leq n-1$ and for all $j\geq i+1$.
	\end{definition}
We refer the reader to
the book \cite{HS} by Swanson and Huneke for more properties of
$d$-sequences.
\begin{proposition} \label{d-seq-star}
The binomial edge ideal of $K_{1,n}$ is of linear type. 
\end{proposition}
\begin{proof}
Let $K_{1,n}$ denote the graph on $[n+1]$ with the edge set $\{ \{i,
n+1\} ~ : ~ 1 \leq i \leq n\}$. We claim that $J_{K_{1,n}}$ is
generated by the $d$-sequence $d_1,d_2,\ldots,d_{n}$, where
$d_i=x_iy_{n+1}-x_{n+1}y_i$. Let $1 \leq i \leq n-1$
and $K_{i+1}$ denote the complete graph on the vertex set $\{1,\ldots,i,n+1\}$. Then,  for $j \geq i+1$,
\[(d_0,d_1,\ldots,d_i):d_{i+1}d_j
=((d_0,d_1,\ldots,d_i):d_{i+1}):d_j=J_{K_{i+1}}:d_j=J_{K_{i+1}},\]
also $(d_0,d_1,\ldots,d_i):d_j=J_{K_{i+1}},$ where the last two
equalities follow from \cite[Theorem 3.7]{FM}.
Therefore, $J_{K_{1,n}}$ is generated by a $d$-sequence. Hence, by
\cite[Corollary 5.5.5]{HS}, $J_{K_{1,n}}$ is of linear type.
\end{proof}
We now prove that in the polynomial ring over an infinite field,
almost complete intersection homogeneous ideals are generated by
$d$-sequences. 

\begin{proposition}\label{d-seq-lemma}
If $I \subset A = \K[t_1, \ldots, t_n]$ is a homogeneous almost
complete intersection, where $\K$ is infinite, then $I$ is generated
by a homogeneous $d$-sequence $f_1, \ldots, f_{h+1}$ such that
$f_1, \ldots, f_h$ is a regular sequence, where $h = \h(I)$.
\end{proposition}
\begin{proof}
Since $I$ is an almost complete intersection ideal, by
\cite[Proposition 5.1(i)]{DK99}, there exists a homogeneous system
of generators $\{f_1, \ldots, f_{h+1}\}$ of $I$ such that $f_1,
\ldots, f_h$ is a regular sequence. Let $J = (f_1, \ldots, f_h)$.
Since $A$ is regular, $J$ is unmixed. It follows from
\cite[Proposition 5.1(ii)]{DK99} and the proof of \cite[Theorem
4.7]{HMV89} that $J : f_{h+1} = J : f_{h+1}^2$.  Therefore, $f_1,
\ldots, f_{h+1}$ is a homogeneous $d$-sequence.
\end{proof}

In the above Lemma, the assumption that $\K$ is infinite is required
in Proposition 5.1 of \cite{DK99}. We assume that $\K$ is infinite for
the following result as well. 

\begin{corollary}\label{d-seq-unicyc-tree}
Let $G$ be a graph on $[n]$.
If $J_G$ is an almost complete intersection ideal,
then $J_G$ is generated by a $d$-sequence. In particular, $J_G$ is of
linear type.
\end{corollary}
\begin{proof}
If $J_G$ is an
almost complete intersection, then it follows from Proposition
\ref{d-seq-lemma} that $J_G$ is generated by a $d$-sequence.
The second assertion that $J_G$ is of linear type is a consequence of 
\cite[Theorem 3.1]{Hu80}.
\end{proof}

If $G$ is a tree or a unicyclic graph of girth $\geq 4$ such that $J_G$
is an almost complete intersection, then one can show that the minimal
generators consisting of the binomials corresponding to the edges
of $G$ form a $d$-sequence.
\begin{remark}\label{aci-remark}
{\em
Suppose $G$ is a tree such that $J_G$ is almost complete intersection.
Then, by Theorem \ref{tree-aci}, $G$ is obtained by adding an edge
between two paths, say $P_{n_1}$ and $P_{n_2}$. Let $e$ denote the
edge between $P_{n_1}$ and $P_{n_2}$. Note that $G\setminus e$ is
the disjoint union of two paths. Assume now that $G$ is a
unicyclic graph with unique cycle $C_m$, $m \geq 4$, such that
$J_G$ is almost complete intersection. Then,  by Theorem
\ref{graph-aci},  $G$ is obtained by adding an edge $e$ between
two vertices of a path. Thus, in both the cases,  $J_{G\setminus
e}$ is complete intersection, by \cite[Corollary 1.2]{her1}. Since
$J_{G\setminus e}$ is a radical ideal,   by Lemma
\ref{radical-lemma}, $J_{G\setminus e} : f_e^2 = J_{G\setminus e}:
f_e$. Hence,   $J_G$ is generated by a $d$-sequence. It may also be
observed that we do not require the assumption that $\K$ is infinite
in this case.
}
\end{remark}

If $G$ is obtained by adding a path each to the vertices of a $C_3$,
then, it can be seen that $J_{G \setminus e}$ is not a complete 
intersection for any edge $e \in E(G)$.  Thus, the binomials
corresponding to the edges of $G$ do not form a $d$-sequence with
first $n-1$ of them forming a regular sequence. 
But at the same time, Proposition
\ref{d-seq-lemma} ensures the existence of such a generating set.
We have not been able to explicitly construct one such.
	
As a consequence of Remark \ref{aci-remark}, we obtain the
defining ideal of the Rees algebra of binomial edge ideals of cycles.
	\begin{corollary}\label{cor-def-eqn} Let $\varphi : S[T_{\{1,n\}},T_{\{i,i+1\}} : i = 1, \ldots,
		n-1] \longrightarrow \R(J_{C_n})$ be the map defined by
		$\varphi(T_{\{i,j\}}) = f_{i,j}t$. 
		The defining ideal of $\R(J_{C_n})$, the kernel of $\varphi$, is minimally generated by 
		\[\left\{f_{i,j}T_{\{k,l\}}-f_{k,l}T_{\{i,j\}}: \{i,j\}\neq \{k,l\} \in
		E(G)\right \}\cup
		\left\{\sum_{k=1}^{n-1} (b_i)_kT_{\{k,k+1\}}-(b_i)_n
		T_{\{1,n\}}:1 \leq i \leq n-1 \right\},\] where $b_i$'s are as
		defined in Theorem \ref{syzygy-cycle}.
	\end{corollary}
	\begin{proof}
		Let 
		\[S(-4)^{n \choose 2} \oplus S(-n)^{n-1}
		\stackrel{\phi}{\longrightarrow} S(-2)^n \longrightarrow J_{C_n}
		\longrightarrow 0\] 
		be the minimal presentation of $J_{C_n}$ given in the proof of Theorem
		\ref{syzygy-cycle}. Since $J_{C_n}$ is of linear type (Remark
		\ref{aci-remark}), it follows from \cite[Exercise 5.23]{HS}
		that the defining ideal of $\R(J_{C_n})$ is generated by $TA$, where
		$A$ is the matrix of $\phi$ and $T = [T_{\{1,2\}}, \ldots, T_{\{n-1,n\}},
		T_{\{1,n\}}]$. Hence,   the assertion follows directly from Theorem
		\ref{syzygy-cycle}.
	\end{proof}

\begin{remark}\label{rem-def-eqn}{\em
Suppose $G$ is a unicyclic graph of girth $m \geq 4$ or a
tree. If $J_G$ is almost complete intersection, then by Remark 
\ref{aci-remark}, $J_G$ is of linear type. Therefore, as in
Corollary \ref{cor-def-eqn}, we can conclude that the defining ideal
of $\R(J_G)$ is generated by $TA$, where $T$ is the matrix consisting
of variables and $A$ is the matrix of the presentation of $J_G$.
Hence,   we obtain a minimal set of generators for the defining ideal of
$\R(J_G)$ by replacing the $e_{\{i,j\}}$'s by $T_{\{i,j\}}$'s in the list of
generators given in the statements in Theorems \ref{syzygy-tree},
\ref{syzygy-unicyclic}. In a similar manner, using Proposition
\ref{d-seq-star} and using a minimal presentation of $J_{K_{1,n}}$,
one can obtain the minimal generators of the defining ideal of the
Rees algebra, $\R(J_{K_{1,n}})$. If $\K$ is infinite, then one can
derive similar conclusions for unicyclic graphs of girth $3$ as well.
}
\end{remark}

\begin{remark}{\em We have shown that if $G$ is a tree with an almost
complete intersection binomial edge ideal $J_G$, then $J_G$ is of
linear type. It would be interesting to know whether binomial edge
ideals of trees, or more generally all bipartite graphs, are of of
linear type. Here we give  an example to show that $J_G$ need not be
of linear type for all bipartite graphs.

\noindent
\begin{minipage}{\linewidth}
	\begin{minipage}{0.72\linewidth}
		Let $G$ be the graph as given on the right.  Then, it can be seen (for example,
		using Macaulay 2 \cite{M2}) that the defining ideal of $J_G$ is not of linear
		type.  If $\delta : S[T_{\{i,j\}} :
		\{i,j\} \in E(G)] \longrightarrow \R(J_G)$ is the map given by $\delta(T_{\{i,j\}}) =
		f_{i,j}t$, then $x_8T_{\{1,6\}}T_{\{3,4\}} - x_6T_{\{1,8\}}T_{\{3,4\}} +
		x_8T_{\{1,4\}}T_{\{3,6\}} - x_4T_{\{1,8\}}T_{\{3,6\}} - x_6 T_{\{1,4\}}T_{\{3,8\}} +
		x_4T_{\{1,6\}}T_{\{3,8\}}$ is a minimal generator of $\ker (\delta)$.
	\end{minipage}
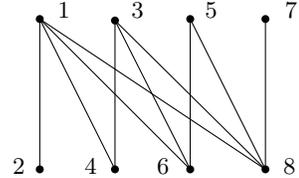
	\begin{minipage}{.25\linewidth}
		\captionsetup[figure]{labelformat=empty}
		\begin{figure}[H]
			
			\begin{tikzpicture}[scale=1]
			
			\draw (-1,3)-- (-1,1);
			\draw (-1,3)-- (0,1);
			\draw (-1,3)-- (1,1);
			\draw (-1,3)-- (2,1);
			\draw (0,1)-- (0,2.98);
			\draw (0,2.98)-- (1,1);
			\draw (0,2.98)-- (2,1);
			\draw (1,3)-- (1,1);
			\draw (1,3)-- (2,1);
			\draw (2,3)-- (2,1);
			\begin{scriptsize}
			\fill  (-1,3) circle (1.5pt);
			\draw (-0.68,3.12) node {$1$};
			\fill  (-1,1) circle (1.5pt);
			\draw (-1.28,1.04) node {$2$};
			\fill  (0,1) circle (1.5pt);
			\draw (-0.32,1.04) node {$4$};
			\fill  (1,1) circle (1.5pt);
			\draw(0.64,1.04) node {$6$};
			\fill  (2,1) circle (1.5pt);
			\draw (2.32,1.04) node {$8$};
			\fill  (0,2.98) circle (1.5pt);
			\draw (0.3,3.12) node {$3$};
			\fill  (1,3) circle (1.5pt);
			\draw (1.28,3.12) node {$5$};
			\fill (2,3) circle (1.5pt);
			\draw (2.34,3.12) node {$7$};
			\end{scriptsize}
			\end{tikzpicture}
		\end{figure}
	\end{minipage}
\end{minipage}
		}
	\end{remark}
	
	It will be interesting to obtain an answer to:
	\begin{question}
		Classify all bipartite graphs whose binomial edge ideals are of linear
		type.
	\end{question}
	Note that the above bipartite graph is not a tree. 
	We have enough experimental evidence to pose the following
	conjecture:
	\begin{con}
		\begin{enumerate}[(a)]
			\item If $G$ is a tree or a unicyclic graph, then $J_G$ is of linear type.
			\item $\R_s(J_{C_n}) = \R(J_{C_n})$, where $\R_s(J_{C_n})$ denote
			the symbolic Rees algebra of $J_{C_n}$.
		\end{enumerate}
	\end{con}

	\bibliographystyle{plain}    
	\bibliography{Reference}
	
\end{document}